\documentclass{amsart}

\usepackage{amssymb}
\usepackage{amsfonts}
\usepackage{amsmath, amsthm}
\usepackage{color}
\usepackage{graphicx}
\usepackage[color,all,cmtip]{xy}
\usepackage{stmaryrd}
\usepackage{transparent}

\newtheorem{theorem}{Theorem}[section]
\newtheorem{lem}[theorem]{Lemma}

\newtheorem{pro}[theorem]{Proposition}

\newtheorem{cor}[theorem]{Corollary}

\numberwithin{equation}{section}

\newcommand{\beq}{\begin{equation}}
\newcommand{\eeq}{\end{equation}}

\newcommand{\Z}{\mathbb Z}
\newcommand{\Q}{\mathbb Q}
\newcommand{\R}{\mathbb R}
\newcommand{\C}{\mathbb C}
\newcommand{\F}{\mathbb F}

\newcommand{\calC}{\mathcal C}
\newcommand{\calZ}{\mathcal Z}
\newcommand{\calO}{\mathcal O}
\newcommand{\calP}{\mathcal P}
\newcommand{\tcalP}{\widetilde{\mathcal P}}
\newcommand{\calR}{\mathcal R}
\newcommand{\calI}{\mathcal I}

\newcommand{\calV}{\mathcal V}
\newcommand{\tcalV}{\widetilde{\mathcal V}}
\newcommand{\calF}{\mathcal F}
\newcommand{\calT}{\mathcal T}
\newcommand{\calD}{\mathcal D}
\newcommand{\tPi}{\widetilde{\Pi}}
\newcommand{\tDelta}{\widetilde{\Delta}}

\renewcommand{\S}{\mathsf{S}}
\newcommand{\D}{\mathsf{D}}
\newcommand{\GL}{\mathsf{GL}}
\newcommand{\gU}{\mathsf{U}}
\newcommand{\gO}{\mathsf{O}}

\DeclareMathOperator{\co}{conv}
\DeclareMathOperator{\ve}{vert}
\DeclareMathOperator{\cone}{cone}
\DeclareMathOperator{\Ind}{Ind}
\DeclareMathOperator{\Res}{Res}
\DeclareMathOperator{\ch}{ch}
\DeclareMathOperator{\Tr}{Tr}
\DeclareMathOperator{\im}{Im}
\DeclareMathOperator{\Id}{Id}
\DeclareMathOperator{\rank}{rank}

\newcommand{\va}{\varphi}
\renewcommand{\b}{\partial}
\newcommand{\vdual}{\widehat v}
\newcommand{\sca}[1]{
	\langle{#1}\rangle
}
\newcommand{\Ps}{{\mathsf P}_{8\cdot 3^s}'}
\newcommand{\Qua}{{\mathsf Q}_8}
\newcommand{\op}{\overline{p}}
\newcommand{\oq}{\overline{q}}
\newcommand{\ox}{\overline{x}}
\newcommand{\oy}{\overline{y}}
\newcommand{\oz}{\overline{z}}
\newcommand{\tv}{\widetilde{v}}
\newcommand{\tV}{\widetilde{V}}
\newcommand{\tx}{\widetilde{x}}
\newcommand{\tu}{\widetilde{u}}
\newcommand{\tz}{\widetilde{z}}
\newcommand{\tvarphi}{\widetilde{\varphi}}
\newcommand{\tR}{\widetilde{R}}
\newcommand{\talpha}{\widetilde{\alpha}}
\newcommand{\tbeta}{\widetilde{\beta}}
\newcommand{\bT}{{\bf T}}
\newcommand{\og}{\overline{g}}
\newcommand{\oh}{\overline{h}}
\newcommand{\ok}{\overline{k}}
\newcommand{\tg}{\widetilde{g}}
\newcommand{\oT}{\overline{T}}
\newcommand{\bu}{\bullet}

\newcommand{\oc}{\overline{c}}
\newcommand{\oalpha}{\overline{\alpha}}
\newcommand{\obeta}{\overline{\beta}}
\newcommand{\ogamma}{\overline{\gamma}}
\newcommand{\odelta}{\overline{\delta}}
\newcommand{\ep}{\epsilon}
\newcommand{\tdelta}{\widetilde{\delta}}
\newcommand{\tep}{\widetilde{\epsilon}}
\newcommand{\ta}{\widetilde{a}}
\newcommand{\tc}{\widetilde{c}}
\newcommand{\oell}{{\hat{\ell}}}

\newcommand{\joint}{\varoast}

\definecolor{light-gray}{gray}{0.75}

\begin{document}

\title{Space forms and group resolutions: the tetrahedral family}

\author{Rocco Chiriv\`\i}
\address{Dipartimento di Matematica e Fisica ``Ennio De Giorgi'', Universit\`a del Salento}
\curraddr{}
\email{rocco.chirivi@unisalento.it}

\author{Mauro Spreafico}
\address{Dipartimento di Matematica e Fisica ``Ennio De Giorgi'', Universit\`a del Salento}
\email{mauro.spreafico@unisalento.it}




\begin{abstract}
The orbit polytope for a finite group $G$ acting linearly and freely on a sphere $\S$ is used to construct a cellularized fundamental domain for the action. A resolution of $\Z$ over $G$ results from the associated $G$--equivariant cellularization of $\S$. This technique is applied to the generalized binary tetrahedral group family; the homology groups, the cohomology rings and the Reidemeister torsions of the related spherical space forms are determined.
\end{abstract}

\maketitle

\section{Introduction}

If $R$ is a ring and $M$ a $R$--module, a resolution of $M$ is an exact sequence of $R$--modules
\[
\xymatrix{
\cdots \ar[r]&F_2\ar[r] & F_1\ar[r] & F_0\ar[r]^{\epsilon}&M\ar[r]&0.
}
\]
Resolutions appear as fundamental objects both in algebra and in topology. In topology, where the ring $R$ is usually the group ring $\Z G$ of the fundamental group $G$ of some space, they represent a basic tool in dealing with the cohomology of groups as well as permit to compute the main algebraic topological  invariants of a space. Unfortunately, to obtain an explicit resolution is in general a very difficult task. A standard technique is to use a simplicial or cellular decomposition of the space, or a $G$--equivariant decomposition of its universal covering. However an explicit decomposition is very hard but for the simplest examples of surfaces and lens spaces.

This approach has been particularly fruitful in the context of a $G$ finite group acting freely on a sphere (see \cite{Mil1} for a list of these groups). These groups have been intensively studied in topology, since they appear as fundamental groups of the spherical space forms, manifolds whose universal covering is a sphere (see for example \cite{DM} and references therein). An explicit knowledge of a ``reasonably simple''  free resolution of $\Z$ over $\Z G$ would carry all interesting algebraic and geometric information; such a resolution for the simplest cases of the cyclic groups and the quaternionic groups has been classically known (see Cartan and Eilenberg \cite[XII.7]{CE}).  However, afterwards, this approach was somehow moved aside in favour to other techniques, mainly because of the intrinsic difficulty in obtaining  suitable simple resolutions (see for example \cite{TZ} for a survey).

Recently, refining the geometric approach introduced by M.~M.~Cohen (see \cite{Coh}), the second author at al. (see \cite{MMS} and \cite{ALOS}) succeeded to find such resolutions for all non abelian groups acting freely and linearly on $\S^3$, except for the generalized binary tetrahedral groups. Indeed a direct approach to the construction of a $G$--equivariant cellular decomposition of the sphere, for $G$ a generalized binary tetrahedral group, turns out to be almost impossible but for the first group of the family, i.e. the binary tetrahedral group (see \cite{ALOS1}).

\smallskip

In this paper, given a finite group $G$ freely acting on a sphere $\S^n\subseteq V$ by a linear representation $\rho:G\longrightarrow\GL(V)$, we construct a $G$--equivariant cellular decomposition in a uniform way. We start by choosing a point $v_0\in\S^n$, consider the orbit $G\cdot v_0$ and its convex hull $\calP$; this is a polytope, called the \emph{orbit polytope}, on which faces the group $G$ acts. The main idea is to use the orbit polytope to derive the cellular decomposition. A similar approach has been used in \cite{EHS} and applied to a new proof of a resolution for finite reflection groups due to De Concini and Salvetti \cite{DCS}.

In our situation $G$ acts freely on the faces of $\calP$ and we prove that there exists a choice of representatives for the facets under this action, whose union projected on $\S^n$ is a fundamental domain.

The combinatorics of the faces of the polytope $\calP$ depend on the choice of the point $v_0$. In order to simplify this combinatorics, finding a somehow natural choice for $v_0$, we locate a as large as possible cyclic subgroup $H$ of $G$ and take for $v_0$ an eigenvector for $H$ in $V$. The restriction $\rho_H=\Res^G_H\rho$ of the representation $\rho$ has the complex line $\Pi_0$ generated by $v_0$ as a summand and on $\Pi_0$ (a real plane) the $H$--orbit of $v_0$ is a polygon $\calP_H$ with $|H|$ vertices.

The next step is to induce $\rho_0:H\longrightarrow\GL(\Pi_0)$, given by $\rho_0(h)=\rho_H(h)_{|\Pi_0}$, to a representation of $G$. This new representation, while being in general of higher dimension, has simpler associated orbit polytope $\widetilde{\calP}$ than the original $V$. Indeed $\widetilde{\calP}$ is the joint of $[G:H]$ copies of the polygon $\calP_H$. One may then recover the original orbit polytope $\calP$ by projecting the polytope $\widetilde{\calP}$ using a criterion to filter the faces which remains faces when projected.

Having constructed a cellularized fundamental domain for $G$ on the sphere we obtain a $G$--equivariant cellularization of the sphere $\S^n$ and we use it to compute a free resolution of $\Z$ as a trivial $G$--module and certain invariants of the spherical space form $\S^n/G$.

\smallskip

In the present paper we apply our technique to the family of the generalized binary tetrahedral groups, denoted by $\Ps$, $s\geq1$, by Milnor in \cite{Mil1}, hence completing the analysis for the groups acting linearly on $\S^3$. It is quite easy to check that the above recalled known results about the other groups for $\S^3$ follow simply by our method. We plan to study the other families of groups for higher spheres in forthcoming papers.

The first tetrahedral group, $s=1$, is somehow different and clearly simpler; so in the sequel we assume $s\geq 2$, however our technique applies as well to the case $s=1$. The free action irreducible representations of $\Ps$ are all of complex dimension $2$, so $\Ps$ acts freely on $\S^3$. The maximal cyclic subgroup of $\Ps$ is of order $2\cdot 3^s$ and has index $4$ in $\Ps$. So the orbit polytope $\widetilde{\calP}$ of the induced representation is in $\R^8$ and we study the projection back to $\calP$ in $\R^4$. We obtain a fundamental domain that is a union of an irregular octahedron and of $(3^s-3)/2$ irregular tetrahedra.

By suitably defining cells on the fundamental domains we are able to give a free resolution $C_\bu$ of $\Z$ over $\Ps$ by modules of ranks $1$ and $4$. Next, guided by the geometry of this cellularization, we define a simpler free resolution $E_\bu$, chain equivalent to $C_\bu$, having modules of ranks $1$ and $2$.

We want to stress a feature of the resolution $E_\bu$. Various approaches may be used to compute an explicit resolution for a finite group. In particular, the technique proposed in \cite{Bra} is quite suitable for a group isomorphic to a semi direct product one of whose factor is a cyclic group, and all groups acting freely on $\S^3$ are of this kind. On the other hand, the resolution obtained in this way has $\Z[\Ps]$--ranks linearly growing with the degree. This is in sharp contrast with our result: the resolution $E_\bu$ is periodic and has minimal $\Z[\Ps]$--ranks as we prove in Corollary \ref{corollary_minimal_ranks}, using the results in \cite{swa} by Swan.

The resolution $E_\bu$ allows the computations of the homology and cohomology groups of the tetrahedral space forms in a straightforward way. Next, using our resolution, we derive the cup product in cohomology; to our best knowledge this ring structure has been already computed only for $s=1$ (see \cite{TZ}).

Finally we present a further application of our resolution by computing the Reidemeister torsions of the generalized binary tetrahedral space forms. We finish our paper by comparing the torsions of the these spherical space forms defined by different free actions.

\smallskip

The paper is organized as follows. The first part, elementary in nature, introduces all topological and combinatorial results we need. In particular in Section \ref{section_polytope} we recall the main definitions and notations for polytopes and in Section \ref{section_joint} we introduce the direct joint $\calP_1\joint\calP_2$ of two polytopes $\calP_1$ and $\calP_2$ describing its faces in terms of the faces of $\calP_1$ and $\calP_2$. Section \ref{section_dualPolygon} is about dual polygons and Section \ref{section_projection} presents a criterion for the faces of a projected polytope. The main Section of the first part is Section \ref{section_fundamentalDomain} in which we see how to construct a fundamental domain using the orbit polytope. Finally in Section \ref{section_inducedPolytope} we prove that the orbit polytope of an induced representation is the direct joint of copies of the orbit polytope of the inducing polytope.

In the second part of the paper we specialize to the generalized binary tetrahedral group family $\Ps$, $s\geq 2$. In Section \ref{section_gen_tetr_group_family} we introduce notations and prove a result about the equivalence of the free action representations of $\Ps$. In the next core Section \ref{section_tetraOrbitPolytope} we describe the orbit polytopes for free actions. In the final Section \ref{section_tetraHomological} we derive all homological consequences and compute the Reidemeister torsions.

\section{Preliminaries about polytopes}\label{section_polytope}

We denote the standard Euclidean scalar product of the two vectors $x,y\in\R^n$ by $\sca{x,y}$ and $|x|=\sqrt{\sca{x,x}}$ is the associated norm. The open ball of radius $r$ and centre $x$ is $B(x,r)=\{y\in\R^n~|~|x-y|<r\}$, we let $\D^n\subseteq\R^n$ be the closed unit ball and $\S^{n-1}\subseteq\R^n$ be its border, the $(n-1)$--dimensional sphere.

The {\it convex hull} $\co(X)$ of a set of points $X\subseteq\R^n$ is the ``smallest'' convex set containing $X$, i.e. it is the intersection of all convex sets that contain $X$
\[
\co(X) \dot= \bigcap \left\{ C\subseteq \R^n ~|~ X\subseteq C, ~C ~{\rm convex} \right\}.
\]
A linear combination $a_1v_1+a_2v_2+\cdots+a_rv_r$ of points $v_1,v_2,\ldots,v_r\in\R^n$ is \emph{convex} if $a_1,a_2,\ldots,a_r\geq0$ and $a_1+a_2+\ldots+a_r=1$. It is clear that $\co(X)$ is the set of all convex linear combinations of the points in $X$. 

A \emph{polytope} $\calP$ is the convex hull of a finite set of points in $\R^n$, the dimension $\dim\calP$ of $\calP$ is the dimension of the affine space generated by $\calP$. A polytope can also be defined as a bounded set given by the intersection of a finite numbers of half spaces. For this and other general properties about polytopes, see \cite{Zie}.

A face of the polytope $\calP$ is the intersection with an affine hyperplane for which the polytope is entirely contained in one of the two half spaces determined by the hyperplane. More precisely, we said that a linear inequality $\varphi(x)\leq c$, where $\varphi$ is a linear functional on $\R^n$ and $c$ a real number, is {\it valid} on $\calP$ if it satisfied by all points $x$ of $\calP$. Then, a {\it face} of $\calP$ is any set of the form
\[
F=\calP\cap \left\{x\in \R^n~|~\varphi(x)=c\right\},
\]
where $\varphi(x)\leq c$ is a valid inequality for $\calP$. We call $\varphi(x)\leq c$ a \emph{defining} inequality for $F$, $\varphi$ a \emph{defining} functional for $F$, and $U=\{x\in\R^n~|~\varphi(x)=c∫\}$ a {\it defining} hyperplane for $F$. Note that, in general, a face has infinite different defining functionals and hyperplanes, and no natural choice among them.

The {\it proper} faces of $\calP$ are the faces $F\neq\calP$. The {\it dimension} $\dim F$ of a face $F$ is the dimension of the affine space generated by $F$, its {\it co-dimension} is $\dim\calP-\dim F$. The faces of dimensions $0$ are called \emph{vertices}, those of dimension $1$ \emph{edges} and those of co-dimension $1$ {\it facets}; the set of all vertices is $\ve(\calP)$. A $d$--face is a face of dimension $d$ and $\calP_d$ is the set of all $d$--faces of $\calP$.

Note that every polytope is the convex hull of its vertices. Also, if $U$ is a defining hyperplane for a face $F$ of $\calP$, then $F=U\cap \calP=\co(U\cap\ve(\calP))$, namely a face is the convex hull of the set of its vertices and is itself a polytope. When we want to stress the vertices of a face $F$ then we write $F=[v_1,v_2,\ldots,v_r]$ where $\{v_1,v_2,\ldots,v_r\}=\ve(F)=\ve(\calP)\cap F$ are the vertices of $F$. Note that the order of the vertices is not important at the moment, but it will be important when we consider oriented faces.

If a polytope contains $0$ as an interior point then any proper face of $\calP$ is defined by an inequality $\varphi(x)\leq1$. We will always assume that this is the case whenever possible.

Despite the non uniqueness of the defining hyperplanes, a facet $F$ has a unique defining hyperplane if $\calP\subseteq\R^n$ has dimension $n$. Further if $0$ is an interior point of $\calP$, there is a unique defining inequality $\varphi(x)\leq1$ for $F$.

We see two simple properties, they will be used in the following sections.
\begin{lem}\label{lemma_polytope_sphere}
Let $\calV$ be a finite subset of the sphere $\S^{n-1}\subseteq\R^n$ and let $\calP=\co(\calV)$, then $\ve(\calP)=\calV$.
\end{lem}
\begin{proof}
By definition $\calP=\co(\calV)$, thus $\ve(\calP)\subseteq\calV$. On the other hand, let $v\in\calV$ and consider the linear functional $\R^n\ni x\longmapsto\sca{x,v}\in\R$. Since $\calV\subseteq\S^{n-1}$, we have $\calP\subseteq\D^n$; hence $\sca{x,v}\leq 1$ for all $x\in\calP$ and $\sca{x,v}=1$ if and only if $x=v$. This shows that $v$ is a vertex of $\calP$.
\end{proof}

\begin{lem}\label{lemma_interior_point}
Let $\calV$ be a finite subset spanning $\R^n$, then a convex combination $\sum_{v\in\calV}\lambda_v v$ is an interior point of $\co(\calV)$ if $\lambda_v>0$ for all $v\in\calV$. 
\end{lem}
\begin{proof} This is clear. \end{proof}

The \emph{cone} on a subset $X$ of $\R^n$ is the set $\cone(X)=\{\lambda x~|~x\in X,\,\lambda\geq0\}$. If $0$ is an interior point of $X$, then $\cone(X)=\R^n$.

\section{Direct joint of polytopes}\label{section_joint}

Given two subsets $X\subseteq\R^n$ and $Y\subseteq\R^m$, their {\it direct joint} $X\joint Y$ is the convex hull $\co((X\times0)\cup(0\times Y))$ in $\R^{n+m}$. The direct joint of two convex sets $X$ and $Y$ is the union of all segments with vertices $(x,0)$ and $(0,y)$ with $x\in X$ and $y\in Y$, or, in formula
\[
X\joint Y = \{(tx,(1-t)y)~|~0\leq t\leq1,\,x\in X,\,y\in Y\}.
\]
We begin with the following simple lemma.
\begin{lem}
If $0$ is an interior point of $X$ and $Y$ then $0$ is an interior point of $X\joint Y.$
\end{lem}
\begin{proof}
Let $\epsilon>0$ be such that $B(0,\epsilon)\subseteq X$ and $B(0,\epsilon)\subseteq Y$. If $(x',y')\in B(0,\epsilon/2)\times B(0,\epsilon/2)$ then $(x',y')=(t(2x'),(1-t)(2y'))$ with $t=1/2$, and $|2x'|=2|x'|<\epsilon$, $|2y'|=2|y'|<\epsilon$; so $(x',y')\in B(0,\epsilon)\joint B(0,\epsilon)\subseteq X\joint Y$. This shows that the neighbour $B(0,\epsilon/2)\times B(0,\epsilon/2)$ of $0$ in $\R^{m+n}$ is contained in $X\joint Y$.
\end{proof}

The direct joint of two polytopes $\calP_1\subseteq\R^n$ and $\calP_2\subseteq\R^m$, is a new polytope in $\R^{m+n}$. Note that $\calP_1$ e $\calP_2$ may intersect in the origin $0$ of $\R^n\times\R^m$, compare \cite[pg. 323]{Zie} for the definition of joint.

In the following proposition we describe the faces of a direct joint of polytopes.
\begin{pro}\label{proposition_joint_faces}
Suppose that $0$ is an interior point for $\calP_1$ and $\calP_2$; then it is an interior point also for $\calP_1\joint\calP_2$. If $\varphi_i(x)\leq 1$, $i=1,2,\ldots,h$, and $\psi_j(y)\leq 1$, $j=1,2,\ldots,k$, are the defining inequalities for the proper faces of $\calP_1$ and $\calP_2$, respectively, then
\[
(\varphi_i,\psi_j)(x,y)\leq 1,\quad i=1,2,\ldots,h,\,j=1,2,\ldots,k
\]
are the defining inequalities for the proper faces of $\calP_1\joint\calP_2$.
\end{pro}
\begin{proof} The first statement is the content of the previous lemma.

Next we show that any functional $(\varphi_i,\psi_j)$, with $1\leq i\leq h$, $1\leq j\leq k$, defines a proper face of $\calP_1\joint\calP_2$. For $x\in\calP_1,\,y\in\calP_2$ we have $(\varphi_i,\psi_j)(tx,(1-t)y)=t\varphi_i(x)+(1-t)\psi_j(y)\leq t + (1-t) = 1$, hence $(\varphi_i,\psi_j)$ is a defining functional for $\calP_1\joint\calP_2$. Again, since $\varphi_i(x)\leq 1$ for all $x\in\calP_1$ and $\psi_j(y)\leq 1$ for all $y\in\calP_2$, the equality $t\varphi_i(x)+(1-t)\psi_j(y) = 1$ holds if and only if $x$ is in the face $F_1$ of $\calP_1$ defined by $\varphi_i$ and $y$ is in the face $F_2$ of $\calP_2$ defined by $\phi_j$. This proves that $(\varphi_i,\psi_j)$ defines the face $F_1\joint F_2$ of $\calP_1\joint\calP_2$ and being $F_1$ and $F_2$ proper, also $F_1\joint F_2$ is proper.

Finally we show that any proper face of $\calP_1\joint\calP_2$ is defined by some $(\varphi_i,\psi_j)$. First of all note that the empty face of the direct joint is the direct joint of the empty face of $\calP_1$, defined by $\varphi_{i_0}$, for a certain $1\leq i_0\leq h$, and of the empty face of $\calP_2$, defined by $\psi_{j_0}$, for a certain $1\leq j_0\leq k$; hence it is defined as stated. So in what follows we consider only non empty faces.

Let $\Phi:\R^{n+m}\longrightarrow\R$ be a functional defining the proper face $F\neq\varnothing$ of $\calP_1\joint\calP_2$. Being $\Phi$ linear, there exist functionals $\varphi:\R^n\longrightarrow\R$ and $\psi:\R^m\longrightarrow\R$ such that $\Phi=(\varphi,\psi)$.

Given $x\in\calP_1$, since the point $(x,0)$ is in $\calP_1\joint\calP_2$, we have $\varphi(x)=\Phi(x,0)\leq 1$; so $\varphi$ is valid on $\calP_1$. In the same way, $\psi$ is valid for $\calP_2$. As proved above $\Phi=(\varphi,\psi)$ defines the face $F=F_1\joint F_2$, with $F_1$ the face of $\calP_1$ defined by $\varphi$ and $F_2$ the face of $\calP_2$ defined by $\psi$.

Now we show that $F_1$ is a proper face of $\calP_1$. By contradiction, let $F_1=\calP_1$ and, being $F\neq\varnothing$ also $F_2\neq\varnothing$, so let $y\in F_2$. Hence
\[
0 = (t\cdot 0 +(1-t)\cdot y)_{|t=1}\in F_1\joint\{y\}\subseteq F_1\joint F_2=F;
\]
but $0$ is an interior point of $\calP$, so $F=\calP$ and this is impossible since we were assuming that $F$ was a proper face. The proof that $F_2$ is proper is analogous.

We conclude that, being $F_1$ and $F_2$ proper faces of $\calP_1$ and $\calP_2$, respectively, then they are defined by certain $\varphi_i$ and $\psi_j$, respectively; so $F$ is defined also by $(\varphi_i,\psi_j)$.
\end{proof}

\begin{cor}\label{corollary_jointFaces} Suppose that $0$ is an interior point for $\calP_1$ and $\calP_2$; then the proper faces of the direct joint polytope $\calP_1\joint\calP_2$ are given by all direct joints $F_1\joint F_2$ with $F_1$ and $F_2$ proper faces of $\calP_1$ and $\calP_2$, respectively.
\end{cor}
\begin{proof} Follows by the previous proposition since in the proof of that proposition, using the notation defined there, we saw that the inequality $(\varphi_i,\psi_j)(x,y)\leq 1$ defines the face $F_1\joint F_2$ where $F_1$ is defined by $\varphi_i(x)\leq 1$ and $F_2$ is defined by $\psi_j(y)\leq 1$.
\end{proof}

\section{Polygons and dual polygons}\label{section_dualPolygon}

A simple computation will be quite useful in the sequel; it is elementary but we prefer to include it here for completeness and reference. Let $n\geq 3$ be an integer, $\theta = 2\pi/n$ and let $v_h=e^{h\theta i}$, for $h=0,1,\ldots,n-1$, be the vertices of the regular $n$--agon $\calP$ in $\C\simeq\R^2$. Let $\vdual_h=e^{(h+\frac{1}{2})\theta i}/cos(\theta/2)$, for $h=0,1,\ldots,n-1$, be the vertices of the regular $n$--agon $\widehat\calP$, we say that $\widehat\calP$ is \emph{dual} to $\calP$. See Figure \ref{figure_n-agon_dual} for the example $n=5$.

\begin{figure}[h]
    \centering
    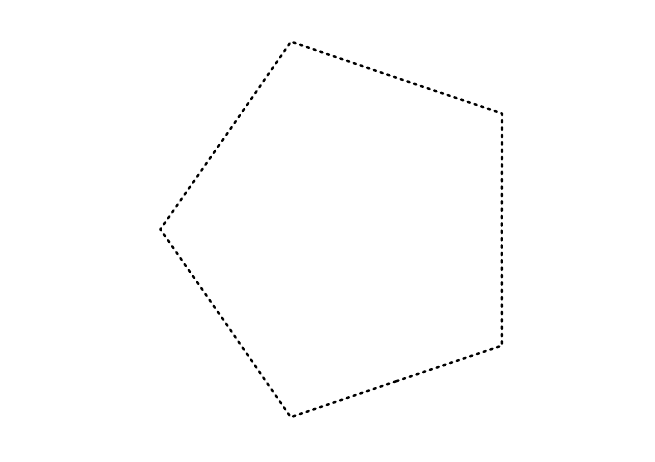
	\caption{The $5$--agon and its dual}
	\label{figure_n-agon_dual}
\end{figure}

This name is due to the following fact: the vertices of $\widehat\calP$ define the edges of $\calP$, while the points on the edges of $\widehat\calP$ define the vertices of $\calP$ and any interior point of $\widehat\calP$ define the empty face of $\calP$. All this is made precise in the following proposition whose easy proof is omitted.
\begin{pro}\label{proposition_polygon_faces}
The inequality $\langle z,\vdual_h\rangle\leq 1$ defines the edge $[v_h,v_{h+1}]$ of $\calP$. The inequality $\langle z,\vdual\rangle\leq 1$ defines the vertices $v_h$ of $\calP$ for any $\vdual = t\vdual_{h-1}+(1-t)\vdual_h$ with $0<t<1$. For any interior point $\vdual$ of $\widehat\calP$, the inequality $\langle z,\vdual\rangle\leq 1$ defines the empty face of $\calP$. Moreover these are all the vectors $\vdual$ such that $\langle z,\vdual\rangle\leq 1$ is valid for $\calP$.
\end{pro}

\section{Projection of polytopes}\label{section_projection}

In this section we briefly study the image of a polytope under a surjective linear map, and in particular we introduce a characterization of the faces of the projected polytope. In our application in later section we will need only the first statement of the following Proposition \ref{proposition_projectedPolytope}, nevertheless we prefer to clarify the relations between the faces of the a polytope and its projection in details.

Let $\pi:\R^n\longrightarrow \R^m$ be a linear map, then the image $\calP=\pi(\widetilde\calP)$ of a polytope $\widetilde\calP$  in $\R^n$ is a polytope in $\R^m$. Indeed $\pi$ send convex linear combinations to convex linear combinations, hence $\calP=\co(\pi(\ve(\widetilde\calP)))$, and in particular, denoting by $\widetilde\calV$ the set of vertices of $\widetilde\calP$ and by $\calV$ that of $\calP$, we have $\calV\subseteq\pi(\widetilde\calV)$.

\begin{lem}
Suppose $\pi:\R^n\longrightarrow \R^m$ is a surjective linear map, $\tcalP$ a polytope in $\R^n$ and $\calP=\pi(\tcalP)$ its projection in $\R^m$ and suppose that $\pi$ is a bijection between the vertices of $\tcalP$ and those of $\calP$. If $\widetilde F$ and $F=\pi(\widetilde F)$ are faces of $\widetilde\calP$ and $\calP$, respectively, then $\ve(F)=\pi(\ve(\widetilde F))$.
\end{lem}
\begin{proof} Since $\widetilde F$ and $F$ are polytopes then $\ve(F)\subseteq \pi(\ve(\widetilde F))$ as remarked above. On the other hand, denoting by $\widetilde\calV$ and $\calV$ the vertices of $\tcalP$ and $\calP$ respectively, $\ve(\widetilde F)=\widetilde V\cap \widetilde F$, hence $\pi(\ve(\widetilde F))=\pi(\widetilde\calV\cap\widetilde F)\subseteq \pi(\widetilde\calV)\cap\pi(\widetilde F)$ and, being $\pi$ a bijection from $\widetilde\calV$ to $\calV$, we find $\pi(\widetilde\calV)\cap\pi(\widetilde F)=\calV\cap F=\ve(F)$.
\end{proof}


\begin{pro}\label{proposition_projectedPolytope} Suppose $\pi:\R^n\longrightarrow \R^m$ is a surjective linear map, $\tcalP$ a polytope in $\R^n$ and $\calP=\pi(\tcalP)$ its projection in $\R^m$.
\begin{itemize}
\item[i)] If $\widetilde F$ is face of $\tcalP$ defined by a linear functional $\widetilde\varphi$ with $\ker\pi\subseteq\ker\widetilde\varphi$ then the projection $\pi(\widetilde F)$ is a face of $\calP$.
\item[ii)] If $F$ is a face of $\calP$ then there exists a face $\widetilde F$ of $\tcalP$ defined by a linear functional $\widetilde\varphi$ with $\ker\pi\subseteq\ker\widetilde\varphi$ such that $\pi(\widetilde F)=F$.
\item[iii)] Suppose moreover that $\pi$ is a bijection between the vertices of $\tcalP$ and those of $\calP$; if the projection $\pi(\widetilde F)$ of the face $\widetilde F$ of $\tcalP$ is a face of $\calP$ then $\widetilde F$ is defined by a linear functional $\widetilde\varphi$ with $\ker\pi\subseteq\ker\widetilde\varphi$.
\end{itemize}
\end{pro}
%
%
\begin{proof}
\begin{itemize}
\item[i)] Let $F=\pi(\widetilde F)$ and let $\widetilde\varphi(\widetilde x)\leq c$ be a defining inequality for $\widetilde F$ with $\ker \pi\subseteq\ker\widetilde \varphi$. Then $\widetilde\varphi$ induces a linear functional $\varphi$ on $\R^m$ such that the following diagram commutes
\[
\xymatrix{
\R^n\ar[r]^{\widetilde\varphi}\ar[d]_\pi&\R.\\
\R^m\ar[ur]_\varphi\\
}
\]
If $x\in\calP$ then $x=\pi(\widetilde x)\in\calP$, for some $\widetilde x\in\widetilde\calP$, hence $\varphi(x)=\widetilde\varphi(\widetilde x)\leq c$, so the inequality $\varphi(x)\leq c$ is valid for $\calP$. Moreover if $x\in F$, then we may assume that $\widetilde x\in\widetilde\calF$ and we have $\varphi(x)=\widetilde\varphi(\widetilde x)=c$; finally, if $x\not\in F$ then $\widetilde x\not\in\widetilde F$, hence $\varphi(x)=\widetilde\varphi(\widetilde x)<c$. This finishes the proof that $\varphi(x)\leq c$ is a defining inequality for $F$, which is a face of $\calP$.

\item[ii)] Suppose now that $F$ is a face of $\calP$ defined by $\varphi(x)\leq c$. The composition $\widetilde\varphi=\pi\varphi$ makes the above diagram commutative and so $\ker\pi\subseteq\ker\widetilde\varphi$.

For $\widetilde x\in\widetilde\calP$ we have $x=\pi(\widetilde x)\in\calP$, hence $\widetilde\varphi(\widetilde x)=\varphi(x)\leq c$ and this shows that $\widetilde\varphi$ is valid for $\widetilde\calP$. Let $\widetilde F$ be the face of $\tcalP$ defined by $\widetilde\varphi(\widetilde x)=c$; using that $\pi$ is surjective we find $\pi(\widetilde F)=F$ by the definition of $\widetilde\varphi$.

\item[iii)] Let $F=\pi(\widetilde F)$, a face of $\calP$ defined by the inequality $\varphi(x)\leq c$. By the previous point, we already know that $F$ is the projection of a face $\widetilde F'$ of $\tcalP$ defined by the inequality $\widetilde\varphi(\widetilde x)\leq c$ where $\widetilde\varphi=\varphi\pi$, and we have $\ker\pi\subseteq\ker\widetilde\varphi$. We want to show that $\widetilde F'=\widetilde F$.

If $\widetilde x\in\widetilde F$ then $x\in F$, thus $\widetilde\varphi(\widetilde x)=\varphi(x)=c$, hence $\widetilde F\subseteq\{\widetilde x\in\R^n~|~\widetilde\varphi(\widetilde x)=c\}=\widetilde F'$.

On the other hand, if $\widetilde x'\in\widetilde F'\setminus \widetilde F$ then $\widetilde x'$ is a convex linear combination
\[
\widetilde x'=\sum_{i=1}^r\lambda_i\widetilde v_i,\quad\lambda_1,\lambda_2,\ldots,\lambda_r\geq0,~\sum_{i=1}^r\lambda_i=1
\]
of the vertices $\ve(\widetilde F')=\{\widetilde{v_1},\widetilde{v_2},\ldots,\widetilde{v_r}\}$ and there exists a vertex $\widetilde{v_{i_0}}\not\in\widetilde F$ such that $\lambda_{i_0}>0$; in other words, a vertex not in $\widetilde F$ appears with positive coefficient in $\widetilde x'$.

Note that $\pi(\widetilde v_{i_0})\not\in F$, since, being $\pi$ bijective on vertices, $\ve(F)=\pi(\ve(\widetilde F))$ as proved in the previous Lemma. Hence $\varphi(\pi(\widetilde v_{i_0}))<c$ and we conclude
\[
\widetilde\varphi(\widetilde x')=\varphi(x)=\sum_{i=1}^r\lambda_i\varphi(\pi(\widetilde v_i))<\left(\sum_{i=1}^r\lambda_i\right)c=c.
\]
But this is impossible since $\widetilde x'\in\widetilde F'$ which is defined by $\widetilde\varphi(\widetilde x)=c$.
\end{itemize}
\end{proof}

\section{Fundamental domain and orbit polytope}\label{section_fundamentalDomain}

Let $G$ be a group acting on a topological space $X$, recall that a \emph{fundamental domain} for this action is a connected closed subset $\calD$ of $X$ such that $X = \bigcup_{g\in G} g\calD$ and $g\calD\cap g'\calD$ has void interior for any pair $g,g'\in G$ with $g\neq g'$.

We are interested in the free actions of a finite group $G$ on the sphere $\S^{n-1}\subseteq \R^{n}$. Fix a point $v_0$ in $S^{n-1}$, and consider the orbit $\calV=Gv_0$; this is a finite set of points in $\S^{n-1}$, and we may consider the \emph{orbit polytope} $\calP=\co(\calV)$ with base point $v_0$. The set of vertices of the orbit polytope is exactly the orbit $\calV$ as follows by Lemma \ref{lemma_polytope_sphere}. Moreover it is clear that $\calP$ is $G$--invariant.

Now we see some preliminary results for the proof of the main theorem of this section.

\begin{lem}
Suppose that the orbit $\calV$ spans $\R^n$, then $0$ is an interior point of $\calP$. In particular the cone over $\calP$ is the whole $\R^n$ and the boundary $\partial\calP$ is homeomorphic to $\S^{n-1}$.
\end{lem}
\begin{proof}
If the point of $\calP$
\[
x = \frac{1}{|G|}\sum_{g\in G}g\cdot v_0=\frac{1}{|G|}\sum_{v\in\calV}v
\]
was not $0$, then $x/|x|$ should be a $G$--invariant point of the sphere; this is impossible acting $G$ freely. So $x=0$ and it is an interior point of $\calP$ be Lemma \ref{lemma_interior_point}.

It is now clear that $\cone(\calP)=\R^n$ since the same is true for a small ball around $0$ contained in $\calP$. Hence the map $\partial\calP\ni x\longmapsto x/|x|\in\S^{n-1}$ is well defined and a homeomorphism.
\end{proof}

\begin{pro}
If $\calV$ spans $\R^n$, then the group $G$ acts freely on the set $\calP_d$ of $d$--faces of the orbit polytope for any $d<\dim\calP$.
\end{pro}
\begin{proof}
Let $F$ be a $d$--face defined by the inequality $\varphi(x)\leq c$ and consider the functional $g\varphi$. By definition $(g\varphi)(x)=\varphi(g^{-1}x)$, thus $(g\varphi)(x)\leq c$ for all $x\in\calP$ by the $G$--invariance of $\calP$; so $g\varphi$ is valid for $\calP$. Moreover $(g\varphi)(x)=1$ if and only if $g^{-1}x\in F$, i.e. if and only if $x\in gF$. This shows that $gF$ is still a face of $\calP$ and it is clear that the action of $G$ does not change the dimension.

Now we show that the action of $G$ on the set $\calP_d$ is free. Given a face $F$ let
\[
b_F = \frac{1}{|\ve(F)|}\sum_{v\in\ve(F)}v.
\]
Since $g\ve(F)=\ve(gF)$, we have
%
\begin{align*}
g b_F  &= g\left(\frac{1}{|\ve(F)|}\sum_{v\in\ve(F)}v\right)\\
 &= \frac{1}{|\ve(gF)|}\sum_{v\in\ve(gF)}v\\
 &= b_{gF}.\\
\end{align*}
Hence if $F$ is fixed by $g\neq e$, we have $g b_F = b_{gF}=b_F$. This forces $b_F=0$ being the action free, otherwise $b_F/|b_F|$ was a point of $\S^{n-1}$ stabilized by $g$. But $0$ is an interior point of $\calP$ by the previous lemma, hence $F$ is defined by an inequality $\varphi(x)\leq1$ for some linear functional $\varphi$ on $\R^n$. So we have
\[
\varphi(b_F)=\frac{1}{|\ve(F)|}\sum_{v\in\ve(F)}\varphi(v)=1
\]
and this shows that $b_F\neq0$. So it is impossible that $gF=F$ and the action is free.
\end{proof}

\begin{cor}\label{l2.2} If $F$ and $F'$ are different proper faces of $\calP$ of the same dimension and $g$ a non trivial element of $G$, then $F\cap gF'$ has void relative interior.
\end{cor}
\begin{proof} By the previous proposition $F$ and $gF'$ are different faces of the polytope $\calP$, hence they intersect in the common boundary if any.
\end{proof}


We are now in a position to prove the main theorem about the orbit polytope and the fundamental domain.

\begin{theorem}\label{theorem_fundamental_domain} Let $G$ be a finite group acting freely by isometries on the sphere $\S^{n-1}\subseteq \R^{n}$, let $v_0$ a fixed point in $\S^{n-1}$ and assume that the orbit $\calV=G\cdot v_0$ spans $\R^n$. Then there exists a system of representatives $F_1,F_2,\ldots,F_r$ for the action of $G$ on the set of facets of the orbit polytope $\calP=\co(\calV)$, such that $F_1\cup F_2\cup\cdots\cup F_r$ is a fundamental domain for $G$ in $\partial\calP$.
\end{theorem}
\begin{proof}
Let $F_1,F_2,\ldots,F_r$ be any set of representatives for the action of $G$ on the facets of $\calP$ and let $\calD=F_1\cup F_2\cup\cdots\cup F_r$. If $x$ is any point in $\partial\calP$ then there exists $g\in G$ and at least an $i$, with $1\leq i\leq r$, such that $gx\in F_i$; thus $\partial\calP=\cup_{g\in G}g\calD$.

Now let $g\in G\setminus\{e\}$. The interior of the set $\calD\cap g\calD$ is the union of the interior of $F_i\cap g F_j$ for $1\leq i,j\leq r$; but this last set is empty by Corollay \ref{l2.2}. Thus the set $\calD\cap g\calD$ has no interior point.

In order to complete the proof we need to show that the representatives $F_1$, $F_2,\ldots,F_r$ of the facets may be chosen so that $\calD$ is connected. Let $F_1$ be any of such representatives and let $F_2,F_3,\ldots,F_k$ be other distinct representatives such that $\calD_0=F_1\cup F_2\cup\cdots\cup F_k$ is connected and $k\leq r$ is maximal with this property. We want to prove that $k=r$.

Let $\calF$ be the family of all facets of $\calP$ which intersect $\calD_0$ non trivially and let $\calD_+$ be the union of all faces in $\calF$. Let $U$ be a neighbourhood of $\calD_0$ in $\partial\calP$ contained in $\calD_+$ (see the Figure \ref{figure_D_U}). If $F$ is in $\calF$ then $F$ is in the orbit of an $F_i$, with $1\leq i\leq k$, by the maximality of $k$; hence denoting by $\pi$ the projection map $\partial P\longrightarrow\partial P / G$ we have $\pi(\calD_0)=\pi(U)$.

\begin{figure}
    \centering
    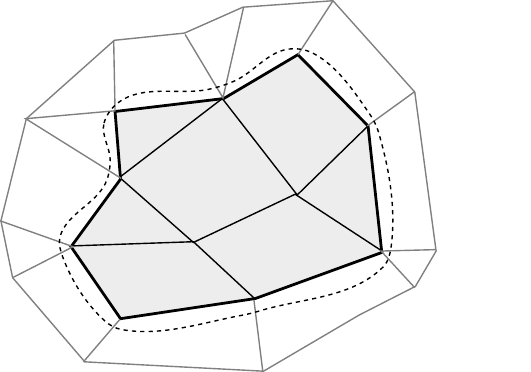
	\caption{The sets $\calD_0$, $\calD_+$ and $U$ as in the proof of the Theorem \ref{theorem_fundamental_domain}}\label{figure_D_U}
\end{figure}

Note that $\pi$ is an open map and, begin $G$ a finite group, it is also closed. So $\pi(\calD_0)=\pi(U)$ is open, since $U$ is open, and it is closed since $\calD_0$ is closed in $\partial\calP$. But $\partial\calP$ and $\partial\calP / G$ are connected; hence $\pi(\calD_0)=\partial\calP / G$, or, in other words, every orbit has a representative in $\calD_0$, i.e. $k=r$.
\end{proof}

\begin{cor}\label{c2.1} If $\calD$ is a fundamental domain for $G$ in $\calP$, then $\cone(\calD)$ is a fundamental domain for the action of $G$ on $\R^n$ and $\cone(\calD)\cap\S^{n-1}$ is a fundamental domain on $\S^{n-1}$.
\end{cor}
\begin{proof} This is clear since, as already noted, $\partial\calP$ is homeomorphic to $\S^{n-1}$ via the map $x\longmapsto x/|x|$.
\end{proof}

In the sequel of this paper we will always identify $\partial\calP$ and $\S^{n-1}$ without any further comment. So, for example, we talk of the fundamental domain $\calD$ of $\S^{n-1}$ as in the previous theorem and corollary, while, properly speaking the domain is $\cone(\calD)\cap\S^{n-1}$.

We will use orbit polytopes to construct fundamental domains but we explicitly remark that the geometry and the combinatorics of an orbit polytope does depend on the base point. This is not surprising since also the notion of fundamental domain is not canonical; for a single (free) action there are plenty of different fundamental domains.

\section{Orbit polytope for induced representations}\label{section_inducedPolytope}

Let $G$ be a finite group, $H$ a subgroup of $G$, $\rho:G\longrightarrow \GL(V)$ a representation and $W\subseteq V$ a $H$--invariant subspace. Recall that $\rho$ is \emph{induced} by $\psi\doteq\rho_{|H}:H\longrightarrow\GL(W)$ if $V=\oplus_{i=1}^r g_iW$, where $g_1,g_2,\ldots,g_r\in G$ are representatives for the quotient $G/H$ and $r=[G:H]$ is the index of $H$ in $G$. Note that the induced representation is unique up to isomorphism of $G$--representations; we denote it by $\Ind_H^G \psi$, or by $\Ind_H^G W$ if the $H$--module structure on $W$ is clear. We have $\dim\Ind_H^G\psi= r\cdot\dim W$.

Now we want to compare the orbit polytope of $H$ in $W$ and that of $G$ in $V=\Ind_H^GW$ in case of real representations. Up to changing the base of $V$ as a vector space over $\R$, we can suppose $W=\R^n\simeq\R^n\times 0\times\cdots\times 0\subseteq\R^{rn}=V$. It is clear that if $G$ acts by isometries and freely on $\S^{nr-1}\subseteq\R^{nr}$ then $H$ acts by isometries and freely on $\S^{n-1}\subseteq\S^{nr-1}$.

Let $v_0\in\S^{n-1}$ be fixed and denote by $\calP_H=\co(H\cdot v_0)$ the orbit polytope of $H$ with base point $v_0$. Then
\[
g_i\cdot\calP_H\subseteq \underbrace{0\times\cdots\times 0}_{i-1}\times\S^{n-1}\times\underbrace{0\times\cdots\times 0}_{r-i}
\]
is isometric to $\calP_H$.

\begin{pro}\label{proposition_induced_polytope}
Let $\calP_G=\co(G\cdot v_0)$ be the orbit polytope of $G$ with base point $v_0$, then
\[
\calP_G=\joint_{i=1}^r g_i\calP_H\simeq\calP_H^{\joint r}.
\]
\end{pro}
\begin{proof}
The orbit $G\cdot v_0$ is given by $g_i h v_0$ as $i$ runs in $1,2,\ldots,r$ and $h$ in $H$. The thesis follows.
\end{proof}

\section{The generalized binary tetrahedral group family}\label{section_gen_tetr_group_family}

Let $s\geq 2$ be an integer and recall that the \emph{generalized binary tetrahedral group} $\Ps$, as denoted by Milnor in \cite{Mil1}, has the following presentation
\[
\Ps = \langle p,q,z~|~p^2=(pq)^2=q^2,\,zpz^{-1}=q,\,zqz^{-1}=pq,\,z^{3^s}=1\rangle.
\]
It is clear that $\Ps$ is already generated, for example, by $p$ and $z$, but this more symmetric presentation is useful. From the presentation one can easily find that $p,q$ and $pq$ all have order $4$ and, denoting by $-1$ the element $p^2=(pq)^2=q^2$, one can prove at once that the following commuting relations hold:
\begin{itemize}
\item[i)] $-1$ is a central element of order $2$ and $z^3$ is a central element of order $3^{s-1}$,
\item[ii)] $qp=-pq$,
\item[iii)] $zp=qz$ and $zq=pqz$,
\item[iv)] $z^2p=pqz^2$ and $z^2q=pz^2$.
\end{itemize}
Using these properties, it is straightforward to prove that
\begin{pro}\label{proposition_equivalence} Each element of $\Ps$ may uniquely be written as $\pm p^m q^n z^k$, with $0\leq m,n\leq 1$ and $0\leq k< 3^s$. In particular $\Ps$ has order $8\cdot 3^s$.
\end{pro}
Another way of expressing above property iii) and iv) is the following: the inner automorphism of $\Ps$ given by conjugation by $z$ acts as a cyclic permutation on the elements $p,q,pq$. We depict this in the following diagram
\[
\xymatrix{
	p\ar[rr]^z & & q\ar[dl]^z\\
	& pq\ar[ul]^z
}
\]

Now we briefly recall how the representations by which $\Ps$ freely acts on a sphere are defined (see \cite{Ste}).  We begin by introducing the complex matrix
\[
Z_0=-\frac{1}{2}\left(
\begin{array}{cc}
1 + i & 1 + i \\
-1 + i & 1 - i \\
\end{array}
\right)\in\gU(2),
\]
it is a unitary matrix and $Z_0^3=\Id$. Let $\theta=2\pi/3^s$, $\zeta=e^{\theta i}$ and let $\ell$ be a positive integer $1\leq \ell<3^s$ prime to $3$; then the assignment
\[
\begin{array}{rcl}
z & \longmapsto & Z = \zeta^\ell Z_0,\\[5pt]
p & \longmapsto & P = \left(\begin{array}{cc}i & 0\\ 0 & -i\\\end{array}\right),\\[10pt]
q & \longmapsto & Q = \left(\begin{array}{cc}0 & 1\\ -1 & 0\\\end{array}\right)\\[5pt]
\end{array}
\]
may be extended to an irreducible unitary representation $\alpha_\ell:\Ps\longrightarrow\gU(2)$, we denote by $V_\ell$ the vector space $\C^2$ with the $\Ps$--module structure of $\alpha_\ell$. Up to isomorphism, these representations define all the free actions of $\Ps$ on $\S^3$; as real representations they have dimension $4$ and are orthogonal, i.e. $\alpha_\ell:\Ps\longrightarrow\gO(4)$. Moreover, up to isomorphism, any free action of $\Ps$ on $\S^{n-1}$ is of type
\[
\alpha_{\ell_0}\oplus\alpha_{\ell_1}\oplus\cdots\oplus\alpha_{\ell_{r-1}}
\]
for certain positive prime to $3$ integers $\ell_0\leq \ell_1\leq \cdots\leq \ell_{r-1}$; in particular $n=4r\equiv 0\pmod 4$.

Now we want to introduce an equivalence relation, weaker than isomorphism, on the set of representations of a group $G$. Let $\rho_1:G\longrightarrow\GL(V_1)$, $\rho_2:G\longrightarrow\GL(V_2)$ be two $G$--representations. We say that $\rho_1$ is \emph{equivalent} to $\rho_2$ if there exists a group automorphism $\varphi$ of $G$ such that $\rho_1\circ\varphi$ is isomorphic to $\rho_2$. This has a certain importance for us since, although the representations $\alpha_\ell$ are not isomorphic, they are all equivalent (see also \cite{Wol}).
\begin{pro}\label{proposition_all_representations_equivalent} The representations $\alpha_\ell$, $1\leq \ell<3^s$, $(\ell,3)=1$, are all equivalent to each other.
\end{pro}
\begin{proof}
We show that the representation $\alpha_1$ is equivalent to $\alpha_\ell$ for any prime to $3$ integer $\ell$ with $1\leq \ell<3^s$. So fix such an integer $\ell$ and consider the assignment
\[
\begin{array}{rcl}
p & \longmapsto & \overline p = \left\{\begin{array}{ll}p & \textrm{if }\ell\equiv+1\pmod3\\ -pq & \textrm{if }\ell\equiv -1\pmod3,\end{array}\right.\\[10pt]
q & \longmapsto & \overline q = \left\{\begin{array}{ll}q & \textrm{if }\ell\equiv+1\pmod3\\ -q & \textrm{if }\ell\equiv -1\pmod3,\end{array}\right.\\[10pt]
z & \longmapsto & \overline z = z^\ell.\\
\end{array}
\]
We prove that such assignment may be extended to an homomorphism $\varphi_\ell$ be showing that the relations defining $\Ps$ are fulfilled by $\op,\oq,\oz$. Suppose first $\ell\equiv+1\pmod3$. Then $\op=p$ and $\oq=q$ and the relations involving only $p$ and $q$ clearly hold also for $\op$ and $\oq$. Now, by conjugation, $\oz$ permutes $\op=p,\oq=q,\op\oq=pq$ as $z$ does since $\ell\equiv+1\pmod 3$, so also the remaining relations hold.

Suppose now $\ell\equiv-1\pmod 3$. Then $\op=-pq$, $\oq=-q$ and $\op\oq=-pq(-q)=-p$; so $\op^2=\oq^2=(\op\oq)^2=-1$. Further, $\oz$ acts by conjugation sending $pq\longmapsto q\longmapsto p\longmapsto pq$ since $\ell\equiv -1\pmod3$; so it does the same on $\op=-pq$, $\oq=-q$ and $\op\oq=-p$. This finishes the proof that $\varphi_\ell$ is an homomorphism.

Moreover $\varphi_\ell$ is surjective: $z\in\im\varphi_\ell$, since $\oz\in\im\varphi_\ell$ and $(\ell,3)=1$, and $p,q\in\im\varphi_\ell$ since $\op,\oq\in\im\varphi_\ell$ and $p,q$ may be written in terms of $\op$ and $\oq$ for both $\ell\equiv\pm1\pmod{3}$. So $\varphi_\ell$ is an automorphism of $\Ps$.

It remains to prove that $\alpha_1\circ\varphi_\ell$ is isomorphic to $\alpha_\ell$. The representation $\alpha_1\circ\varphi_\ell$ defines a free action of $\Ps$ since $\alpha_1$ does; so $\alpha_1\circ\varphi_\ell$ is isomorphic, as a complex representation, to $\alpha_h$ for a suitable positive integer $1\leq h<3^s$ with $(3,h)=1$. In order to find such a $h$ we compare the characters of $\alpha_1\circ\varphi_\ell$ and $\alpha_h$.

We have $\ch(\alpha_1\circ\varphi_\ell)(z)=\ch(\alpha_1)(z^\ell)=\zeta^\ell\Tr(Z_0^\ell)$ and, being $Z_0$ of order $3$, we find $\ch(\alpha_1\circ\varphi_\ell)(z)=-\zeta^\ell$ since both $Z_0$ and $Z_0^{-1}$ have trace $-1$. In the same way, $\ch(\alpha_h)(z)=-\zeta^h$. So, we conclude $h=\ell$ since the two isomorphic representations $\alpha_1\circ\varphi_\ell$ and $\alpha_h$ must have the same characters and $\zeta$ is a primitive $3^s$--root of unity.
\end{proof}

\section{The generalized binary tetrahedral orbit polytopes}\label{section_tetraOrbitPolytope}

Our aim in this core section is the description of an orbit polytope for a free action of $\Ps$ on the $3$--dimensional sphere $\S^3\subseteq\C^2$. It turns out that, to our best understanding, this problem is quite combinatorially and geometrically complicated to deal with directly. So we take a somehow longer way, passing to a higher dimensional representation, by which we are able to conclude.

We consider a cyclic subgroup $H$ of order $2\cdot 3^s$ of $\Ps$, an $H$--invariant complex line and the $4$--dimensional complex induced representation from $H$ to $\Ps$. In this higher dimensional representation the orbit polytope is simple to describe, it is the direct joint of four $2\cdot 3^s$--polygons thanks to the result of Section \ref{section_inducedPolytope}. Then we return to the original orbit polytope in $\C^2$ by projecting and picking out the faces in $\C^4$ that verify the criterion about projected polytope in Section \ref{section_projection}.

\smallskip

Let us fix a free action representation $\alpha_\ell:\Ps\longrightarrow\gU(V_\ell)$ in $\C^2=\R^4$ as in the previous section. The largest order of an element of $\Ps$ is $2\cdot 3^s$, the order of $x\doteq -z$. Let $H=\langle x\rangle$ be the subgroup generated by $x$, recall that $\theta=2\pi/3^s$, $\zeta=e^{i\theta}$ and note that $x$ acts by the matrix
\[
X = -\zeta^\ell Z_0=\frac{1}{2}\left(\begin{array}{cc} 1 + i & 1 + i\\ -1 + i & 1 - i\\\end{array}\right)
\]
whose eigenvalues are $\lambda = \zeta^\ell(1-\sqrt{-3})/2 = \zeta^{\ell-3^{s-1}}$ and $\lambda'=\zeta^\ell(1+\sqrt{-3})/2=\zeta^{\ell+3^{s-1}}$. So denoting by $v_0\in\C^2$ an eigenvector of eigenvalue $\lambda$ for $X$, the complex line $\Pi_0=\C\cdot v_0$ is $H$--invariant. The induced representation $\tV_\ell=\Ind_H^{\Ps}\Pi_0$ is $4$-dimensional complex since $H$ has index $4$ in $\Ps$. 

\begin{lem}\label{lemma_decomposition} The $\Ps$--representation $\tV_\ell$ decomposes as $V_\ell\oplus V_{\ell-2\cdot 3^{s-1}}$; in particular it defines a free action of $\Ps$ on the sphere $\S^7\subseteq\C^4$.
\end{lem}
\begin{proof} For a $2\cdot 3^s$--root of unity $\eta$, denote by $\C_\eta$ the unique $1$--dimensional $H$ module whose $x$ action is multiplication by the scalar $\eta$; moreover let $\lambda_\ell^\pm$ be $2\cdot 3^s$--root of unity $\zeta^\ell(1\pm\sqrt{-3})/2$. First note that $\Res^{\Ps}_H V_h = \C_{\lambda^-_h}\oplus\C_{\lambda^+_h}$, for any $1\leq h < 3^s$ and $(h,3)=1$. Next we use the Frobenius reciprocity (see, for example, \cite{Ser}) and compute
\begin{align*}
\langle \tV_\ell, V_h\rangle_{\Ps} &= \langle\Ind_H^{\Ps}\C_{\lambda_\ell^-},V_h\rangle_{\Ps}\\
  &= \langle\C_{\lambda_\ell^-},\Res^{\Ps}_H V_h\rangle_H\\
  &= \langle\C_{\lambda_\ell^-},\C_{\lambda^-_h}\oplus\C_{\lambda^+_h}\rangle_H.\\
\end{align*}
Hence $\langle \tV_\ell, V_h\rangle_{\Ps}$ is $1$ if and only if either $\lambda^-_h=\lambda^-_\ell$ or $\lambda^+_h=\lambda^-_\ell$. In the first case $h=\ell$ while in the second case we have $\lambda^+_h=\zeta^{h+3^{s-1}}=\zeta^{\ell-3^{s-1}}=\lambda^-_\ell$ and we find $h=\ell-2\cdot 3^{s-1}$.
\end{proof}

\subsection{The orbit polytope in $\R^8$}

The first step now is the description of the orbit polytope $\tcalP=\co(\Ps\cdot \tv_0)\subseteq\S^7\subseteq\tV=\C^4=\R^8$ with base point $\tv_0=v_0\in\tPi_0\subseteq \tV$ where $\tPi_0$ is the plane $\Pi_0$ as a subset of $\tV$.

The elements $g_0=1,\,g_1=p,\,g_2=q,\,g_3=pq$ are a system of representatives for $\Ps/H$, hence defining $\tPi_j=g_j\tPi_0$, for $j=0,1,2,3$, we have $\tV=\tPi_0\oplus\tPi_1\oplus\tPi_2\oplus\tPi_3$. The orbit $H\cdot\tv_0$ on the real plane $\tPi_0$ is the set $\tcalV_0$ of vertices of a regular $2\cdot 3^s$--polygon since $\alpha_{\ell|H}$ is a free action and $H$ is a cyclic group of order $2\cdot 3^s$.

Note that $w\longmapsto \lambda w$ is a rotation of
\[
\frac{\ell - 3^{s-1}}{2\cdot 3^s}2\pi
\]
radians in the real plane $\tPi_0$, since $\lambda=\zeta^{\ell-3^{s-1}}$. So taking $\ell=\oell \doteq 1+3^{s-1}$ we have a rotation of $\pi/3^s$ radians, i.e. the centre angle of a $2\cdot 3^s$--agon; being such $\oell$ prime to $3$, the representation $\alpha_\oell$ gives a free action. Fixing $\ell$ does not arm generality since all free actions are equivalent by Proposition \ref{proposition_equivalence}; so in the rest of this section, where not stated otherwise, the representation $\alpha_\ell$ is fixed with $\ell=\oell$, $X$ is a rotation of $\pi/{3^s}$ radians in the plane $\Pi_0$ and, of course, the same is true for the action of $x$ in $\tPi_0$.

Let $\tcalP_0=\co(\tcalV_0)$, a $2\cdot 3^s$--polygon in the plane $\tPi_0$; we denote its vertices by the corresponding elements of the group, so $x^h$ is the $h$--th vertex of $\tcalP_0$ starting from $\tv_0$ and counting counter-clockwise, i.e. $x^h$ is the vertex $x^h\cdot \tv_0$ 

The full orbit $\Ps\cdot\tv_0$ in $\R^8$ is given by $\tcalV_0\cup\tcalV_1\cup\tcalV_2\cup\tcalV_3$ where $\tcalV_j=g_j\tcalV_0\subseteq\tPi_j$, for $j=0,1,2,3$. The polygon $\co(\tcalV_j)\subseteq\tPi_j$ is denoted by $\tcalP_j$ and its vertices are denoted by the corresponding group elements: $g_jx^h$ is the vertex $g_jx^h\cdot\tv_0$, $h=0,1,\ldots,2\cdot 3^s-1$; moreover we set also $\tv_j=g_j\cdot\tv_0$. Note that $\tv_0$, $\tv_1$, $\tv_2$, $\tv_3$ is a complex basis for $V$ while $\tv_j,i\tv_j$ is a real basis for $\tPi_j$ for $j=0,1,2,3$. We define the scalar product $\tV\times\tV\ni(\tu,\tv)\longmapsto\sca{\tu,\tv}\in\R$ as the standard Euclidean scalar product with respect to this real basis of $\tV$.

In what follows we will need many times to compute the action of various $\Ps$ elements on the vertices of the polytope $\tcalP$; of course this is just group element multiplication given the way we denote the vertices. Anyway we summarize the action of the generators (and also of $pq$) of $\Ps$ in the following diagrams
\[
\xymatrix@!C=5pt@!R=9pt{
     & 0\ar@(ul,ur)^-{+1}\ar@[light-gray]@{-}[dl]\ar@[light-gray]@{-}[dr]\ar@[light-gray]@{-}[dd]\\
   1\ar@/^1pc/[rr]^{+1}\ar@[light-gray]@{-}[rr] &   & 2\ar@/^1pc/[dl]^{+1}\ar@[light-gray]@{-}[dl]\\
     & 3\ar@/^1pc/[ul]^{+1}\ar@[light-gray]@{-}[ul]\\
     & {}\save[]+<0cm,0.25cm>*{x}\restore\\
}
\quad
\xymatrix@!C=5pt@!R=9pt{
     & 0\ar@/^0.5pc/[dl]^{+}\ar@[light-gray]@{-}[dl]\ar@[light-gray]@{-}[dr]\ar@[light-gray]@{-}[dd]\\
   1\ar@/^0.5pc/[ur]^{-}\ar@[light-gray]@{-}[rr] &   & 2\ar@/_0.5pc/[dl]_{+}\ar@[light-gray]@{-}[dl]\\
     & 3\ar@/_0.5pc/[ur]_{-}\ar@[light-gray]@{-}[ul]\\
     & {}\save[]+<0cm,0.25cm>*{p}\restore\\
}
\quad
\xymatrix@!C=5pt@!R=9pt{
     & 0\ar@/_0.5pc/[dr]_{+}\ar@[light-gray]@{-}[dl]\ar@[light-gray]@{-}[dr]\ar@[light-gray]@{-}[dd]\\
   1\ar@/_0.5pc/[dr]_{-}\ar@[light-gray]@{-}[rr] &   & 2\ar@/_0.5pc/[ul]_{-}\ar@[light-gray]@{-}[dl]\\
     & 3\ar@/_0.5pc/[ul]_{+}\ar@[light-gray]@{-}[ul]\\
     & {}\save[]+<0cm,0.25cm>*{q}\restore\\
}
\quad
\xymatrix@!C=5pt@!R=9pt{
     & 0\ar@/_0.5pc/[dd]_(0.6){+}\ar@[light-gray]@{-}[dl]\ar@[light-gray]@{-}[dr]\ar@[light-gray]@{-}[dd]\\
   1\ar@/^3.5pc/[rr]^{+}\ar@[light-gray]@{-}[rr] &   & 2\ar@/^3.5pc/[ll]^{-}\ar@[light-gray]@{-}[dl]\\
     & 3\ar@/_0.5pc/[uu]_(0.6){-}\ar@[light-gray]@{-}[ul]\\
     & {}\save[]+<0cm,0.25cm>*{pq}\restore\\
}
\]
where, for $g=x,\,p,\,q,\,pq$ as in a diagram, an arrow $\xymatrix@1{i\ar[r]^{+1} & j}$ means that $g\cdot g_i x^h = g_j x^{h+1}$, an arrow $\xymatrix@1{i\ar[r]^{+} & j}$ means that $g\cdot g_i x^h = g_j x^h$ and, finally, an arrow $\xymatrix@1{i\ar[r]^{-} & j}$ means that $g\cdot g_i x^h = -g_j x^h=g_j x^{h+3^s}$.

By Proposition \ref{proposition_induced_polytope} the polytope $\tcalP$ is the direct joint of four polygons $\tcalP_0\joint\tcalP_1\joint\tcalP_2\joint\tcalP_3$, isometric to $\tcalP_0^{\joint 4}$. Hence, by Corollary \ref{corollary_jointFaces}, a proper face of $\tcalP$ is the direct joint of four (possibly empty) proper faces, one for each polygon $\tcalP_j$. Since a polygon has only the empty face, vertices and edges we see that a face of $\tcalP$ is obtained by picking $0$, $1$ or $2$ consecutive vertices on each polygon and taking the direct joint of these vertices.

We set up a notation for certain faces of $\tcalP$ we need in the sequel. First we introduce some $5$--simplexes
\[
\begin{array}{rcl}
\tDelta_5(h_0,h_1,h_2,-) & = & [x^{h_0},x^{h_0+1}]\joint[px^{h_1},px^{h_1+1}]\joint[qx^{h_2},qx^{h_2+1}]\\
 & = & [x^{h_0},x^{h_0+1},px^{h_1},px^{h_1+1},qx^{h_2},qx^{h_2+1}],
\end{array}
\]
the joint of the $h_0$--th edge of the polygon $\tcalP_0$, of the $h_1$--th edge of the polygon $\tcalP_1$ and of the $h_2$--th edge of the polygon $\tcalP_2$. And similarly for $3$--simplexes (i.e. tetrahedron)
\[
\begin{array}{rcl}
\tDelta_3(h_0,h_1,-,-) & = & [x^{h_0},x^{h_0+1}]\joint[px^{h_1},px^{h_1+1}]\\
 & = & [x^{h_0},x^{h_0+1},px^{h_1},px^{h_1+1}],
\end{array}
\]
the joint of the $h_0$--th edge of the polygon $\tcalP_0$ and of the $h_1$--th edge of the polygon $\tcalP_1$. Similar notations apply to any combinations of three or two of the four polygons $\tcalP_0,\tcalP_1,\tcalP_2,\tcalP_3$ by moving the symbol(s) ``--'' in the other positions.

\subsection{The orbit polytope in $\R^4$}

We are now in a position to project $\tcalP$ to $\calP=\co(\Ps\cdot v_0)\subseteq V=\C^2=\R^4$. So let $v_j=g_j\cdot v_0\in\C^2$, for $j=0,1,2,3$, and define the map
\[
\pi:\C^4\longrightarrow\C^2
\]
by $\C$--linearly extending
\[
\tv_j\longmapsto v_j,\quad \textrm{for }j=0,1,2,3.
\]
Crucial is the following
\begin{lem}
The projection $\pi$ is the unique $\Ps$--equivariant map $\tV\longrightarrow V$ such that $\pi(\tv_0)=v_0$.
\end{lem}
\begin{proof} Since $V=V_\oell$, $\tV=V_{\oell}\oplus V_{\oell-3^{s-1}}$ and these last two representations are not isomorphic, the space of $\Ps$--equivariant linear map $\tV\longrightarrow V$ is $1$--dimensional by the Schur Lemma (see \cite{Ser}). As already noted in the proof of Lemma \ref{lemma_decomposition}, $\Res^{\Ps}_H V=\C_{\lambda_\oell^-}\oplus\C_{\lambda_\oell^+}$ using the notation introduced there; hence the vector $\tv_0\in\C_{\lambda_\oell^-}\subseteq\Res^{\Ps}_H\tV$ must be sent to a non-zero scalar multiple of $v_0$ by a $\Ps$--equivariant projection $\varphi:\tV\longrightarrow V$ since $H\subseteq\Ps$. So there exists and is unique such a $\varphi$ with $\varphi(\tv_0)=v_0$; moreover we also have $\varphi(\tv_j)=\varphi(g_j\tv_0)=g_j\varphi(\tv_0)=g_j v_0=v_j$. We conclude that $\pi = \varphi$ since the two linear maps coincide on the basis $\tv_0,\tv_1,\tv_2,\tv_3$ of $\tV$.
\end{proof}

We use the following notational convention: given any object $\widetilde A$ related to $\tV$ we denote by $A$ its projection via $\pi$ to $V$; for example $\Pi_j=\pi(\tPi_j)$, for $j=0,1,2,3$. The previous Lemma assures that any relation in terms of $\Ps$ among objects $\widetilde A$ and $\widetilde B$ is still valid among $A$ and $B$; for example $\Pi_j=g_j\Pi_0$, for $j=0,1,2,3$.

We explicitly note that the above diagrams giving the actions of $x,p,q$ and $pq$ on the vertices of $\tcalP$ apply to the vertices of $\calP$ too, by the $\Ps$--equivariance of the map $\pi$. Furthermore, we denote also the vertices of $\calP$ by the elements of the group $\Ps$ identifying $g$ and $g\cdot v_0$; with this notation the projection $\pi$ from the vertices of $\tcalP$ to those of $\calP$ is simply $\pi(g)=g$.

In the following steps we need some common notations that we fix now: let $\phi=\pi/3^s$ be the centre angle of a $2\cdot 3^s$--polygon and let $\omega=e^{i\pi/3}=(1+\sqrt{-3})/2$ be a primitive sixth root of unity in $\C$. In order to apply Proposition \ref{proposition_projectedPolytope} to the pair $\tcalP\stackrel{\pi}{\longrightarrow}\calP$ we need to know when a linear functional $\tvarphi$ on $\tV$ has kernel containing the kernel of $\pi$. This is the content of the following proposition.

\begin{pro}\label{proposition_admissible_functional}
Let $\tu$ be a vector in $\tV$ with coordinates $(\tz_0,\tz_1,\tz_2,\tz_3)$ with respect to the basis $\tv_0,\tv_1,\tv_2,\tv_3$. Then the linear functional $\tV\ni\tx\stackrel{\tvarphi}{\longmapsto}\sca{\tx,\tu}\in\R$, has the property $\ker\pi\subseteq\ker\tvarphi$ if and only if
\[
\left\{
\begin{array}{l}
\tz_2 = w^2\tz_0 - w\tz_1\\
\tz_3 = w\tz_0 + w^2\tz_1.\\
\end{array}
\right.
\]
\end{pro}
\begin{proof}
As a first step we prove that the following system of equations holds in $V$
\[
\left\{
\begin{array}{l} 
\sqrt{-3}v_0+v_1+v_2+v_3=0\\
-v_0+\sqrt{-3} v_1-v_2+v_3=0.\\
\end{array}
\right.
\]
It is straightforward to check that
\[
Z_0=-\frac{1}{2}(1+P+Q+PQ),
\]
hence
\[
Z=\zeta^\oell Z_0=-\frac{1}{2}\zeta^\oell(1+P+Q+PQ).
\]
Applying this to $v_0$, we obtain the equation
\[
(1-\sqrt{-3})v_0=v_0 + Pv_0+Qv_0+PQv_0,
\]
and, by the definition of $v_0,v_1,v_2,v_3$, the equation
\[
\sqrt{-3}v_0 + v_1 + v_2 + v_3 = 0.
\]
Finally, applying $g_1$ to this last equation we find the other equation in the system. (One can also apply $g_2$ and $g_3$ and obtain a more symmetric system with four equations of rank $2$.)

Now we prove that $\ker\pi$ is generated by the following two vectors
\begin{align*}
\tR_0 &= \sqrt{-3}\tv_0+\tv_1+\tv_2+\tv_3,\\
\tR_1 &= -\tv_0+\sqrt{-3} \tv_1-\tv_2+\tv_3.\\
\end{align*}
Indeed, let $K$ be the vector subspace of $\tV$ generated by $\tR_0$ and $\tR_1$; by the previous system of equations fulfilled by $v_0,v_1,v_2,v_3$ we have $K\subseteq\ker\pi$ since $\pi$ sends each $\tv_j$ in $v_j$ by definition, for $j=0,1,2,3$. Moreover $K$ has dimension $2$ since $\tR_0$ and $\tR_1$ are clearly linearly independent. So $K=\ker\pi$ using $\dim\ker\pi=\dim\tV-\dim V=2$.

Finally note that $\ker\pi$ is a complex subspace of $\tV$ since $\pi$ is a $\C$--linear map. So $(\ker\pi)^\perp$ may also be defined via the standard Hermitian scalar product $\tV\times\tV\ni(\tv,\tv')\longmapsto\sca{\tv,\tv'}_\C\in\C$ with respect to the basis $\tv_0,\tv_1,\tv_2,\tv_3$ of $\tV$. Hence the equations to prove are obtained from $\sca{\tu,\tR_h}_\C=0$ for $h=0,1$.
\end{proof}

We introduce now certain projected faces from $\tcalP$ to $\calP$. For an integer $h$ let
\begin{align*}
\calO(h) &= \Delta_5(h,h+2\cdot 3^{s-1},h+3^{s-1},-)\\
 &= \pi(\tDelta_5(h,h+2\cdot 3^{s-1},h+3^{s-1},-))\\
 &= [x^h,x^{h+1},px^{h+2\cdot 3^{s-1}},px^{h+2\cdot 3^{s-1} + 1},qx^{h+3^{s-1}},qx^{h+3^{s-1} + 1}],\\
 \end{align*}
this is the convex hull of three edges of the three polygons $\calP_0,\calP_1,\calP_2$ in the three different planes $\Pi_0,\Pi_1,\Pi_2$; we will see it is a $3$--dimensional polytope in $V=\R^4$. Clearly $\calO(h)$ depends only on the residue class of $h$ modulo $2\cdot 3^s$. We call any polytope $g\cdot\calO(h)$, with $g\in\Ps$ and $h$ integer, an \emph{admissible octahedron} (see the subsequent Proposition \ref{proposition_octahedron_h_faces} for this name).

For $h$ and $k$ integers let
\begin{align*}
\calT(h,k) &= \Delta_3(h,k,-,-)\\
 &= \pi(\tDelta_3(h,k,-,-))\\
 &= [x^h,x^{h+1},px^k,px^{k+1}],
\end{align*}
this is the convex hull of two edges in $\calP_0,\calP_1$, a tetrahedron in $V=\R^4$. Also $\calT(h,k)$ depends only on the residue class of $h$ and $k$ modulo $2\cdot 3^s$. We call any polytope $g\cdot\calT(h,k)$, with $g\in\Ps$ and $h,k\in\Z$ such that $h+3^{s-1}<k<h+2\cdot 3^{s-1}$, an \emph{admissible tetrahedron}.

\begin{pro}\label{proposition_octahedra_are_facet} Any admissible octahedron is a facet of $\calP$.
\end{pro}
\begin{proof} Let $c = 1/\cos(\phi/2)$ and
\[
\begin{array}{l}
\tz_0=c\cdot e^{i(h + \frac{1}{2})\phi},\\
\tz_1=c\cdot e^{i(h + 2\cdot 3^{s-1} + \frac{1}{2})\phi},\\
\tz_2=c\cdot e^{i(h + 3 ^{s-1} + \frac{1}{2})\phi},\\
\tz_3=0\\
\end{array}
\]
and consider the functional $\tV\ni\tx\stackrel{\tvarphi}{\longmapsto}\langle\tx,\tz\rangle\in\R$, where $\tz=(\tz_0,\tz_1,\tz_2,\tz_3)$. By Proposition \ref{proposition_joint_faces} and Proposition \ref{proposition_polygon_faces} the functional $\tvarphi$ defines the face $\tDelta_5(h,h+2\cdot 3^{s-1},h+3^{s-1},-)$ of $\tcalP$. Moreover, note that $e^{3^{s-1}\phi i}=\omega$ and it is easy to verify that $\tz$ fulfils the condition of Proposition \ref{proposition_admissible_functional}. So $\calO(h)=\pi\tDelta_5(h,h+2\cdot 3^{s-1},h+3^{s-1},-)$ is a facet of $\calP$. Clearly also $g\cdot\calO(h)$, for any $g\in\Ps$, is a facet of $\calP$.
\end{proof}

In the proof of the next proposition we need the positivity of a certain function, we see this in the following lemma.
\begin{lem}\label{lemma_positive_function}
The function
\begin{align*}
\left[0,\frac{\pi}{3}\right] &\longrightarrow \R\\
\alpha &\longmapsto \frac{1}{2}\cos\frac{\alpha}{2} - \cos\left(\frac{2\pi}{3} + \alpha\right) - 1\\
\end{align*}
vanishes in $0$ and is positive in $(0,\pi/3]$.
\end{lem}
\begin{proof} Let $f$ be the function defined above. It is clear that $f(0)=0$. Moreover $f'(\alpha)=-a(\alpha)+b(\alpha)$, where $a(\alpha)=\sin(\alpha/2)/4$ and $b(\alpha)=\sin(2\pi/3 + \alpha)$ with $a(\alpha)$ increasing in $[0,\pi/3]$ and $b(\alpha)$ decreasing in the same interval.

Since $f'(0)=\sqrt{3}/2>0$ while $f'(\pi/3)=-\sin(\pi/6)/4<0$, the function $f'(\alpha)$ has exactly one zero, say $\alpha_0$, in $[0,\pi/3]$. This implies that $f$ is increasing till $\alpha_0$ and decreasing from $\alpha_0$ to $\pi/3$; but since $f(\pi/3)=\cos(\pi/6)/2>0$ we have the claim about the positivity of $f$.
\end{proof}

\begin{pro}\label{proposition_tetrahedra_are_facet} Any admissible tetrahedron is a facet of $\calP$.
\end{pro}
\begin{proof}
Let $c = 1/\cos(\phi/2)$ and
\[
\begin{array}{l}
\tz_0=c\cdot e^{i(h + \frac{1}{2})\phi},\\
\tz_1=c\cdot e^{i(k + \frac{1}{2})\phi},\\
\tz_2=\omega^2\tz_1 - \omega\tz_2,\\
\tz_3=\omega\tz_1+\omega^2\tz_2.\\
\end{array}
\]
The functional $\tV\ni\tx\stackrel{\tvarphi}{\longmapsto}\langle\tx,\tz\rangle\in\R$, where $\tz=(\tz_0,\tz_1,\tz_2,\tz_3)$, fulfils the condition in Proposition \ref{proposition_admissible_functional} by definition. Moreover if we show that $|\tz_2|,|\tz_3|<1$ then $\tz_2$ and $\tz_3$ are internal points of the dual polygons of $\tcalP_2$ and $\tcalP_3$ and, by Proposition \ref{proposition_joint_faces} and Proposition \ref{proposition_polygon_faces}, the functional $\tvarphi$ defines the face $\tDelta_3(h,k,-,-)$ of $\tcalP$. So we conclude that $\calT(h,k)=\pi(\tDelta_3(h,k,-,-))$ is a facet of $\calP$; clearly also all $g\calT(h,k)$, for $g\in\Ps$, are facet of $\calP$.

Now we prove that $|\tz_2|<1$. Let $d=k-h$ and note that
\[
|\tz_2| = |c\omega^2e^{i(h+\frac{1}{2}\varphi)}(1 + e^{i(d\phi + \frac{2\pi}{3})})| = 2c(1 + \cos(d\phi + \frac{2\pi}{3})).
\]
The hypothesis on $h$ and $k$ implies that $\pi/3 <d\phi<2\pi/3$, so $\pi<d\phi + 2\pi/3\leq 4\pi/3-\phi$ and we find $\cos(d\phi + 2\pi/3)\leq\cos(\phi + 2\pi/3)$. Hence
\[
|\tz_2|=2c(1 + \cos(d\phi + \frac{2\pi}{3}))\leq 2c(1+\cos(\phi + \frac{2\pi}{3}))
\]
and so $|\tz_2|<1$ using the previous Lemma since $c=1/\cos(\phi/2)$ and $\phi=\pi/3^s<\pi/3$.

We proceed analogously for proving $|\tz_3|<1$. We have
\[
|\tz_3| = |c\omega e^{i(h+\frac{1}{2}\varphi)}(1 + e^{i(d\phi + \frac{\pi}{3})})| = 2c(1 + \cos(d\phi + \frac{\pi}{3})).
\]
The hypothesis on $h$ and $k$ implies that $2\pi/3+\phi \leq d\phi + \pi/3<\pi$, hence $\cos(d\phi+\pi/3)\leq\cos(\phi+2\pi/3)$ and we conclude as above using the previous Lemma.
\end{proof}

In order to simplify next computations we see the following proposition first.
\begin{pro}\label{proposition_conjugate_octaedhra}
All the admissible octahedra are in the same orbit under $\Ps$. In particular $x^{3^{s-1}}qx$ maps each vertex
\[
x^h,x^{h+1},px^{h+2\cdot 3^{s-1}},px^{h+2\cdot 3^{s-1} + 1},qx^{h+3^{s-1}},qx^{h+3^{s-1} + 1}
\]
of $\calO(h)$ to the vertex
\[
qx^{h+3^{s-1}+1},qx^{h+3^{s-1} + 2},x^{h+1},x^{h+2},px^{h+2\cdot 3^{s-1}+1},px^{h+2\cdot 3^{s-1} + 2}
\]
of $\calO(h+1)$, respectively.
\end{pro}
\begin{proof}
Just compute
\[
\begin{array}{rcl}
\calO(h) & = & \Delta_5(h,h+2\cdot 3^{s-1},h+3^{s-1},-)\\
 & \stackrel{x}{\longmapsto} & \Delta_5(h+1,-,h+2\cdot 3^{s-1}+1,h+3^{s-1}+1)\\
 & \stackrel{q}{\longmapsto} & \Delta_5(h+2\cdot 3^{s-1}+1+3^s,h+3^{s-1}+1,h+1, -)\\
 & \stackrel{x^{3^{s-1}}}{\longmapsto} & \Delta_5(h + 1,h+2\cdot 3^{s-1} + 1,h+3^{s-1} + 1,-)\\
 & = & \calO(h+1).
\end{array}
\]
This shows that the admissible octahedra of type $\calO(h)$, $h\in\Z$, are in the same orbit; hence the same is clearly true for all admissible octahedra.

For the second statement one check at once, by the above computation, that $x^{3^{s-1}}qx$ maps the vertices accordingly to the two lists in the Proposition.
\end{proof}
\begin{pro}\label{proposition_octahedron_h_faces} Any admissible octahedron is an (irregular) octahedron as in Figure \ref{figure_octahedron_h}.
\begin{figure}
    \centering
    \def\svgwidth{0.5\columnwidth}
    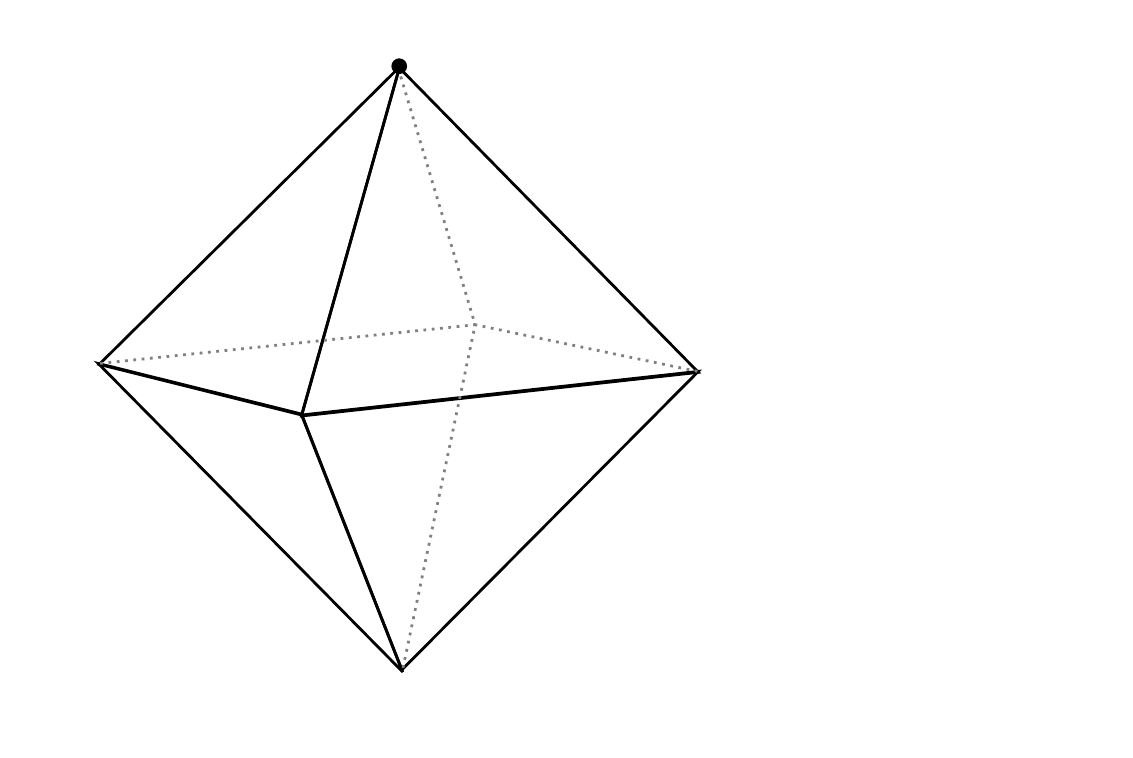
  \caption{The octahedron $\calO(h)$.}\label{figure_octahedron_h}
\end{figure}

\end{pro}
\begin{proof} We show that each $2$--face of the octahedron in the Figure is an actual $2$--face of $\calO(h)$. First if $g=x^{3^{s-1}}qx$ as in the previous proposition, then $g^h$ sends $\calO(0)$ to $\calO(h)$ mapping the vertexes accordingly as they are listed in the previous proposition; so we can assume $h=0$ since faces are sent to faces by the group action on $\calP$ and so does also the permutation on the vertexes of the proposition from $2$--faces of the Figure for $h$ to the $2$--faces of the same Figure for $h+1$.

We begin by showing that six of the eight $2$--faces in the figure are faces also of some tetrahedron, so they are in the border of $\calO(0)$. It is immediate to see that
\begin{align*}
[1,x,p x^{2\cdot 3^{s-1}}]&\subseteq \b \calT(0,2\cdot 3^{s-1}-1),\\
[x,p x^{2\cdot 3^{s-1}},p x^{2\cdot 3^{s-1}+1}]&\subseteq \b \calT(1,2\cdot 3^{s-1}),
\end{align*}
and $\calT(0,2\cdot 3^{s-1}-1)$, $\calT(1,2\cdot 3^{s-1})$ are facet of $\calP$ by Proposition \ref{proposition_tetrahedra_are_facet}.

We explicitly note the two following transformations by element of $\Ps$: $x\calT(h,k)=[x^{h+1},x^{h+2},qx^{k+1},qx^{k+2}]$ and $x^{-1}q\calT(h,k)=[px^{h-1},px^h,qx^{k+3^s-1},qx^{k+3^s}]$.

In particular for $h=-2$ and $k=3^{s-1}-1$ the polytope $\calT(h,k)$ is a facet of $\calP$ by Proposition \ref{proposition_tetrahedra_are_facet} and so the $2$--face $[1,qx^{3^{s-1}},qx^{3^{s-1}+1}]$ is also a face of $[x^{-1},1,qx^{3^{s-1}},qx^{3^{s-1}+1}]=x\calT(h,k)$.

In the same way, for $h=-1$, $k=3^{s-1}$ the polytope $\calT(h,k)$ is a facet of $\calP$ and the $2$--face $[1,x,qx^{3^{s-1}+1}]$ is also a face of $x\calT(h,k)$.

For $h=2\cdot 3^{s-1}+1$ and $k=4\cdot 3^{s-1}$ the polytope $\calT(h,k)$ is a facet of $\calP$ and the $2$--face $[px^{2\cdot 3^{s-1}},px^{2\cdot 3^{s-1}+1},qx^{3^{s-1}}]$ is also a face of $x^{-1}q\calT(h,k)$.

In the same way, for $h=2\cdot 3^{s-1}+2$, $k=4\cdot 3^{s-1}+1$ the polytope $\calT(h,k)$ is a facet of $\calP$ and the $2$--face $[px^{2\cdot 3^{s-1}+1},qx^{3^{s-1}},qx^{3^{s-1}+1}]$ is also a face of $x^{-1}q\calT(h,k)$.

It remains to deal with the two faces: $\calT_1=[1,px^{2\cdot 3^{s-1}}, qx^{3^{s-1}}]$ and  $\calT_2=[x,px^{2\cdot 3^{s-1}+1}, qx^{3^{s-1}+1}]$. These faces can not be faces of tetrahedra of type $g\calT(h,k)$, with $g\in\Ps$, since their vertices lay in three different planes. However, we can use the following general fact: each $1$--face in the boundary of a three dimensional polytope belongs to the boundary of exactly two $2$--faces.

We apply this to $\calO(0)$ considering the $1$--face $L=[1,p x^{2\cdot 3^{s-1}}]$. Observe that $L$ belongs to the boundary of the triangle $[1,x, p x^{2\cdot 3^{s-1}}]$, that we have already seen to be in the boundary of $\calO(0)$. Thus, $L$ should belong to the boundary of another triangle in the boundary of $\calO(0)$, namely there should exist a vertex $v\not= 1,x, p x^{2\cdot 3^{s-1}}$ of $\calO(0)$, such that $\calT=[v,L]$ is in the boundary of $\calO(0)$.

Now, $v$ can not be $p x^{2\cdot 3^{s-1}+1}$, since in this case the edge $[p x^{2\cdot 3^{s-1}}, p x^{2\cdot 3^{s-1}+1}]$ would be in the boundary of three triangles: $\calT$ and two faces of the octahedron already found above. Similarly, $v$ can not be $q x^{ 3^{s-1}+1}$, since in this case the edge $[1,q x^{ 3^{s-1}+1}]$ would be in the boundary of three triangles: $\calT$ and two faces of the octahedron found above. This shows that $v=q x^{3^{s-1}}$, and proves that $\calT$ belongs to the boundary of $\calO(0)$.

In the same way, using $L'=[px^{2\cdot 3^{s-1}+1},qx^{3^{s-1}+1}]$, one can prove that the last $2$--face $[x,px^{2\cdot 3^{s-1}+1}, qx^{3^{s-1}+1}]$ is in the boundary of $\calO(0)$. 
\end{proof}

\begin{pro}\label{proposition_octahedron_border} Each $2$--face of an admissible octahedron is either a face of some admissible tetrahedron or of some other admissible octahedron.
\end{pro}
\begin{proof} We can assume that the admissible octahedra if $\calO(0)$. In the proof of the previous Proposition we saw that four $2$--faces of $\calO(0)$ are faces also of some tetrahedron. For the remaining two ones note that $[1,px^{2\cdot 3^{s-1}}, qx^{3^{s-1}}]$ is a face of $\calO(-1)$ while $[x,px^{2\cdot 3^{s-1}+1}, qx^{3^{s-1}+1}]$ is a face of $\calO(1)$.   
\end{proof}

\begin{pro}\label{proposition_tetrahedron_border} Each $2$--face of an admissible tetrahedron is either a face of some admissible octahedron or of some other admissible tetrahedron.
\end{pro}
\begin{proof} We can assume that the admissible tetrahedron is $\calT(h,k)$ with $h+3^{s-1}<k<h+2\cdot 3^{s-1}$. We consider the four faces of $\calT(h,k)$: 
\begin{align*}
\calT_1&= [x^h, x^{h+1}, p x^k],\\
\calT_2&= [x^h, x^{h+1}, p x^{k+1}],\\
\calT_3&= [x^h,  p x^k, p x^{k+1}],\\
\calT_4&= [x^{h+1}, p x^k, p x^{k+1}].
\end{align*}

If $k\neq h+3^{s-1}+1$, then $\calT(h,k-1)$ is a facet of $\calP$ by Proposition \ref{proposition_tetrahedra_are_facet} and $\calT_1$ is a face also of this tetrahedron. If otherwise $k=h+3^{s-1}+1$, then $\calT_1 = [x^h, x^{h+1}, p x^{h+3^{s-1}+1}]$, and, since $[x^{h+3^{s-1}+1}, px^{h+3^{s}}, p x^{h+3^{s}+1}]$ is a $2$--face of $\calO(h + 3^{s-1})$ by Proposition \ref{proposition_octahedron_h_faces}, $\calT_1$ is a face of $p\calO(h + 3^{s-1})$.

If $k\neq h+2\cdot 3^{s-1}-1$, then $\calT(h,k+1)$ is a facet of $\calP$ by Proposition \ref{proposition_tetrahedra_are_facet} and $\calT_2$ is a face also of this tetrahedron. If otherwise $k = h+2\cdot 3^{s-1}-1$, then $\calT_2$ belongs to the boundary of the octahedron 
$\calO(h)$ by Proposition \ref{proposition_octahedron_h_faces}.

If $k\neq h+2\cdot 3^{s-1}-1$, then $\calT(h-1,k)$ is a facet of $\calP$ by Proposition \ref{proposition_tetrahedra_are_facet} and $\calT_3$ is a face also of this tetrahedron. If otherwise $k = h+2\cdot 3^{s-1}-1$, then $\calT_3$ belongs to the boundary of the octahedron $\calO(h-1)$ by Proposition \ref{proposition_octahedron_h_faces}.

If $k\neq h+3^{s-1}+1<k$, then $\calT(h+1,k)$ is a facet of $\calP$ by Proposition \ref{proposition_tetrahedra_are_facet} and $\calT_4$ is a face also of this tetrahedron. If otherwise $k = h+3^{s-1}+1$, then $\calT_4=[x^{h+1},px^{h+3^{s-1}+1},px^{h+3^{s-1}+2}]$ and, since $[x^{h+3^{s-1}+1},x^{h+3^{s-1}+2},px^{h+3^s+1}]$ is a $2$--face of $\calO(h+3^{s-1}+1)$ by Proposition \ref{proposition_octahedron_h_faces}, $\calT_4$ is a face of $p\calO(h + 3^{s-1} + 1)$.
\end{proof}

We finally conclude with a description of all the facets of $\calP$.
\begin{theorem}\label{theorem_calP_facet} The facets of $\calP$ are the admissible octahedra and the admissible tetrahedra.
\end{theorem}
\begin{proof} By Proposition \ref{proposition_octahedra_are_facet} and Proposition \ref{proposition_tetrahedra_are_facet} all admissible octahedra and all admissible tetrahedra are facets of $\calP$; let $\calF$ be the set of all such facets. Being $\b\calP$ homeomorphic to $\S_3$, it is connected and if $\cup\calF$ was not the whole border of $\calP$ then there should exist a facet in $\calF$ whose border was not contained in $\cup\calF$. But this is impossible by Proposition \ref{proposition_octahedron_border} and Proposition \ref{proposition_tetrahedron_border}.
\end{proof}

\subsection{The fundamental domain}

We are now in a position to describe a fundamental domain for the action of $\Ps$ on $\S^3$. We introduce some further notations: for $g,g'\in\Ps$ and $h,k$ integers, let
\[
\calT_{g,g'}(h,k) = [gx^h,gx^{h+1},g'x^k,g'x^{k+1}].
\]
\begin{theorem}\label{theorem_tetrahedral_fundamental_domain}
The union of the octahedron $\calO(0)$ (yellow) and the tetrahedra
\[
\begin{array}{lll}
\calT_{1,p}(h,2\cdot 3^{s-1}), & 1\leq h\leq  (3^{s-1}-1)/2 & \quad (red),\\[5pt]
\calT_{q,1}(h,0), & 3^{s-1} + 1\leq h\leq 3^{s-1} + (3^{s-1}-1)/2 & \quad (green),\\
\calT_{p,q}(h,3^{s-1}), & 2\cdot 3^{s-1} + 1\leq h\leq 2\cdot 3^{s-1} + (3^{s-1}-1)/2 & \quad (blue)\\[5pt]
\end{array}
\]
is a fundamental domain for the action of $\Ps$ on $\S^3$. (See Figure \ref{figure_fundamental_domain}.)
\begin{figure}
    \centering
    \def\svgwidth{0.6\columnwidth}
    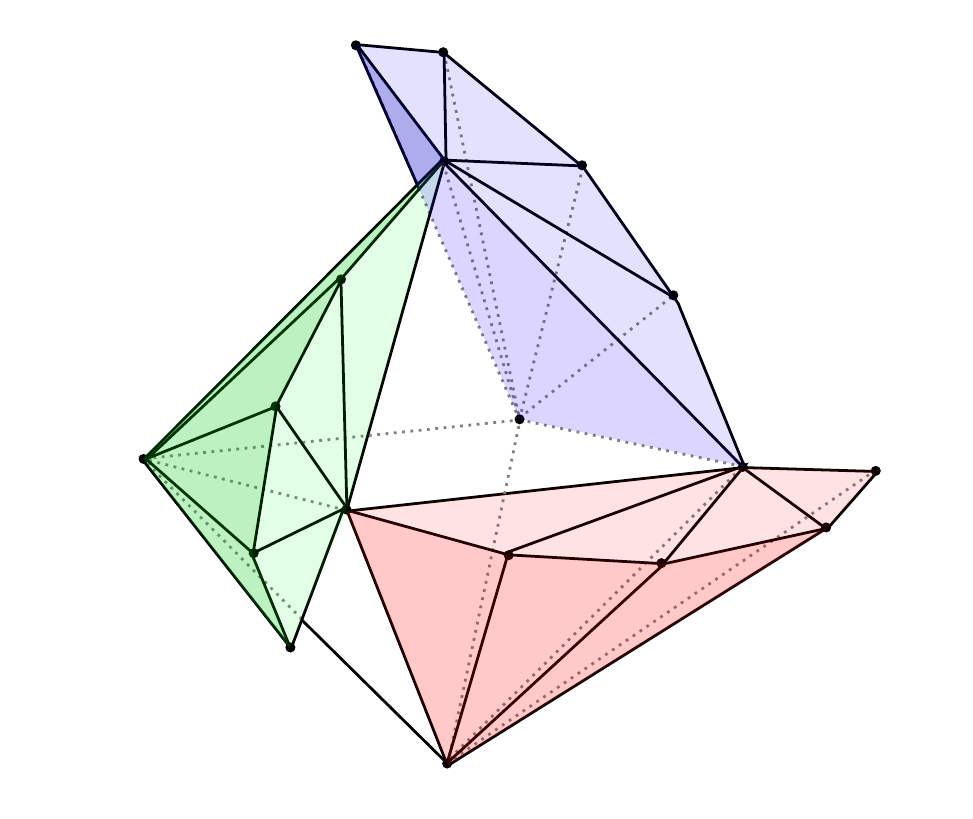
  \caption{The fundamental domain (for $s=3$).}\label{figure_fundamental_domain}
\end{figure}
\end{theorem}
\begin{proof}
By Theorem \ref{theorem_fundamental_domain} a fundamental domain for the action of $\Ps$ on $\S^3$ is given by the union of a, suitably chosen, set of representatives for the $\Ps$--orbits on the set of facets of $\calP$. We begin by noting that the admissible octahedra are all in the same orbit by Proposition \ref{proposition_conjugate_octaedhra}, hence we may chose $\calO(0)$ as a representative. The proof for the tetrahedra is quite more involved.

Let $\bT$ be the set of all tetrahedra in the statement of the Theorem and paint them red, green or blue as indicated (see also Figure \ref{figure_fundamental_domain}). For this proof, set also $a=(3^{s-1}-1)/2$

Since we know that any facet of $\calP$ is an admissible octahedron or is an admissible tetrahedron, we proceed in three steps: first we show that any tetrahedron in $\bT$ is admissible; next, we show that the number of tetrahedra in $\bT$ is the number of orbits of $\Ps$ on the admissible tetrahedra; finally, we show that all tetrahedra in $\bT$ are in different orbits.\\
{\bf Step 1.} Since $h+3^{s-1} < 2\cdot 3^{s-1} < h + 2\cdot 3^{s-1}$ for all $1\leq h\leq a$, any red tetrahedron $\calT_{1,p}(h,2\cdot 3^{s-1})=\calT(h,2\cdot 3^{s-1})$ is admissible.\\
Note that $x^{-1}T_{q,1}(h,0)=\calT(-1,h-1)$ and $-1+3^{s-1} < h-1 < -1 + 2\cdot 3^{s-1}$ for any $3^{s-1}+1\leq h\leq 3^{s-1}+a$. Hence any green tetrahedron is admissible.\\
Finally, for blue tetrahedra, note that $x^{-3^s}pqx\calT_{p,q}(h,3^{s-1})=\calT(3^{s-1}+1,h+1)$ and, further, $3^{s-1}+1+3^{s-1}<h+1<3^{s-1}+1+2\cdot 3^{s-1}$ for any $2\cdot3^{s-1}+1\leq h\leq 2\cdot 3^{s-1}+a$; hence also all such tetrahedra are admissible.\\
{\bf Step 2.} Now we show that $\bT$ contains the correct number of tetrahedra. Let $\calT$ denotes the set of all the admissible tetrahedra. First, observe that $P_{8\cdot 3^s}'$ acts transitevely on the set of the unordered pairs of planes $\Pi_j$, $j=0,1,2,3$ and, for such action, the stabilizer of $\{\Pi_0,\Pi_1\}$ is the subgroup $K=\langle x^3, p\rangle$. Let 
\[
\calT_{1,p}=\{T(h,k)~|~h+3^{s-1}<k<h+2\cdot 3^{s-1},\, 0\leq h\leq 2\cdot 3^s-1\},
\]
be the set of the admissible tetrahedra with vertices on the planes $\Pi_0$ and $\Pi_1$. The element $x^3$ clearly maps $\calT_{1,p}$ onto itself. We have also that $p\calT(h,k)=\calT(k+3^s,h)$, and this last tetrahedron is again admissible; thus $p$ maps $\calT_{1,p}$ onto itself. This shows that $K$ acts on $\calT_{1,p}$.

On the other hand, if $g\in\Ps\setminus K$ and $T\in\calT_{1,p}$, then $g\cdot T\not\in\calT_{1,p}$ since $g\cdot T$ has vertices on the planes $g\Pi_0,g\Pi_1$ and $\{g\Pi_0,g\Pi_1\}\neq\{\Pi_0,\Pi_1\}$ being $K$ the stabilizer of the latter pair of planes.

So we conclude that the orbits of $\Ps$ on $\calT$ are in bijection with the orbits of $K$ on $\calT_{1,p}$. Since $\Ps$ has no fixed point, all actions are free and we have
\[
\left| \calT/\Ps\right|=\left| \calT_{1,p}/K\right|=\frac{|\calT_{1,p}|}{|K|}=\frac{2\cdot 3^{s}(3^{s-1}-1)}{4\cdot 3^{s-1}}
=3\frac{3^{s-1}-1}{2}=|\bT|,
\]
where the last equality is an immediate check.\\
{\bf Step 3.} In this last step we prove that all tetrahedra in $\bT$ are in different $\Ps$--orbits. Suppose that $\tg\in\Ps$ maps $T=\calT_{g,g'}(h,k)$ in $\oT=\calT_{\og,\og'}(\oh,\ok)$ with these tetrahedra both in $\bT$. First of all, note that either
\[
\textrm{(A)}\quad\left\{
\begin{array}{rcl}
\tg \cdot g x^h & = & \og x^{\oh}\\
\tg \cdot g' x^k & = & \og' x^{\ok}\\
\end{array}
\right.
\]
or
\[
\textrm{(B)}\quad\left\{
\begin{array}{rcl}
\tg g \cdot x^h & = & \og' x^{\ok}\\
\tg g' \cdot x^k & = & \og x^{\oh}.\\
\end{array}
\right.
\]
In the rest of the proof we exploit the subgroup $G=\langle p, q, x^{3^{s-1}}\rangle$ of $\Ps$; by the defining relations of $\Ps$, this is a normal subgroup of $\Ps$. Now we separately consider the two cases (A) and (B).\\
Case (A). If $T$ and $\oT$ have the same colour, then $g=\og$, $g'=\og'$, $k=\ok$ and we find $\tg=e$; hence $T=\oT$. So suppose that $T$ and $\oT$ have different colours and note that $g'x^k$, $\og'x^{\ok}$ are elements of $G$, so, by the second equation in (A), we find $\tg\in G$. Recall that we have defined $g_0=e$, $g_1=p$, $g_2=q$ and $g_3=pq$; so there exists $t\in\{0,1,2,3\}$ and $0\leq u\leq 5$ such that $\tg=g_tx^{u\cdot 3^{s-1}}$. Since $x^{3^{s-1}}$ is in the centre of $\Ps$, we have
\[
\left\{
\begin{array}{rcccccl}
\tg g x^h & = & g_tx^{u\cdot 3^{s-1}} g x^h & = & g_t g x^{h + u\cdot 3^{s-1}} & = & \og x^{\oh}\\
\tg g' x^k & = & g_tx^{u\cdot 3^{s-1}} g' x^k & = & g_t g' x^{k + u\cdot 3^{s-1}} & = & \og' x^{\ok}.\\
\end{array}
\right.
\]
In particular
\[
\left\{
\begin{array}{rcl}
g_t g & = & \pm\og\\
g_t g' & = & \pm\og',\\
\end{array}
\right.
\]
since $g_t,g,g',\og,\og'\in\{1,p,q,pq\}$ and $\{\pm1,\pm p,\pm q,\pm pq\}$ is a subgroup of $\Ps$.

So if $T$ is red and $\oT$ is green we have
\[
\left\{
\begin{array}{rcl}
g_t\cdot 1 & = & \pm q\\
g_t\cdot p & = & \pm1;\\
\end{array}
\right.
\]
whereas if $T$ is red and $\oT$ is blue we have
\[
\left\{
\begin{array}{rcl}
g_t\cdot 1 & = & \pm p\\
g_t\cdot p & = & \pm q\\
\end{array}
\right.
\]
and, finally, if $T$ is green and $\oT$ is blue
\[
\left\{
\begin{array}{rcl}
g_t\cdot q & = & \pm p\\
g_t\cdot 1 & = & \pm q.\\
\end{array}
\right.
\]
All these three systems are impossible. We conclude that in case (A) we can only have $T=\oT$.\\
Case (B). Since $g'x^k,\og'x^{\ok}\in G$ we find
\[
\left\{
\begin{array}{rcl}
\tg\cdot g x^h & \equiv & 1 \pmod{G}\\
\tg\cdot 1 & \equiv & \og x^{\oh} \pmod{G};\\
\end{array}
\right.
\]
hence $\og x^{\oh} \cdot g x^h\in G$. Moreover let $t\in\{0,1,2,3\}$ be such that $x^{\oh}\cdot g = g_t x^{\oh}$. Then $\og x^{\oh} \cdot g x^h = \og g_t x^{\oh + h}$ and we conclude $\oh + h\equiv 0\pmod{3^{s-1}}$.

Finally, since the tetrahedron $T$ and $\oT$ are in $\bT$ we have $1\leq h,\oh\leq a\pmod{3^{s-1}}$ and so $2\leq \oh + h \leq 2a=3^{s-1}-1\pmod{3^{s-1}}$; this shows that $\oh + h\neq0\pmod{3^{s-1}}$ and so case (B) is impossible.

We have hence showed that all tetrahedra in $\bT$ are in different orbits and the Theorem is proved.
\end{proof}

Recall that in all this section we have always considered, also implicitly,  the representation $\alpha_\oell$ with $\oell=1+3^{s-1}$. Now we want to consider a generic free action representation $\alpha_\ell$ for $1\leq \ell < 3^s$ and $(\ell,3)=1$.

In the proof of Proposition \ref{proposition_all_representations_equivalent} we have defined certain homomorphism $\varphi_\ell$ of $\Ps$ such that $\alpha_1\circ\varphi_\ell$ is isomorphic to $\alpha_\ell$ as a $\Ps$--representations. So the element $x_\ell\doteq \varphi_\ell^{-1}\varphi_\oell(x)$ acts on $V_{\alpha_\ell}$ as $x$ does on $V_{\alpha_\oell}$, in particular it has an eigenvalue giving a rotation of $\pi/3^s$ in the real plane generated by the corresponding eigenvector. Hence the description of the geometry of the fundamental domain for $V_\oell$ is valid also in $V_\ell$ for a generic $\ell$ as explained below.

\begin{cor}\label{corollary_fundamental_generic} The fundamental domain of the previous Theorem \ref{theorem_tetrahedral_fundamental_domain} for $V_\oell$ is a fundamental domain also for $V_\ell$, $\ell$ any integer prime with $3$, once we replace $x$, $p$ and $q$ by $x_\ell = \varphi_\ell^{-1}\varphi_\oell(x)$, $p_\ell = \varphi_\ell^{-1}\varphi_\oell(p)$ and $q_\ell = \varphi_\ell^{-1}\varphi_\oell(q)$, respectively.
\end{cor}


\section{Homological results for generalized binary tetrahedral groups}\label{section_tetraHomological}

In the previous section we have constructed a simplicial decomposition of the sphere $S^3$ equivariant  with respect to the action $\alpha_\ell$ of the group $P'_{8\cdot 3^s}$. This is clear since we may decompose the octahedron in four tetrahedra. By definition this induces a $\Delta$--simplicial decomposition (see for example \cite{Hat}) of the quotient spherical space form. However, instead of using this decomposition, and the associate chain complex, for homology calculation, it is possible and much more convenient to derive a simpler equivariant decomposition, considering blocks of simplices and lowering the number of cells.  This new decomposition will be an equivariant CW--decomposition. This is the purpose of the first part of this section. In the second part, we compute homology and cohomology group, and we determine the structure of the group cohomology ring. We compute also the Reidemeister torsion for the tetrahedral spherical space forms.
\subsection{Cellular chain complex}

As a first step we want to use a cellular decomposition of the space form $X=\S^3/\Ps$ to construct a $\Ps$--invariant cellular decomposition of $\S^3$, defining a complex $C_\bu$ of $\Z\Ps$--modules.

We begin by defining certain cells via the fundamental domain described in Theorem \ref{theorem_tetrahedral_fundamental_domain}; so we fix the free action $\alpha_\oell$, with $\oell = 1+3^{s-1}$ as seen in the previous section. With reference to Figure \ref{figure_cells}, we define a $2$--cell by listing the vertices of its border in the positive order, i.e. using the anticlockwise order with respect to the normal vector to the border of the fundamental domain.

\noindent{\bf Notation.} In all this section $a$ is the integer $3^{s-1}$.

Let $c_3$ be the whole domain and define the following $2$--cells
\[
\begin{array}{rcl}
c_{2,1} & = & [x, x^2, \dots, x^{\frac{a+1}{2}-1}, x^{\frac{a+1}{2}},\\
 &  & p x^{2a+1}, p x^{2a+2}, \dots, p x^{2a+\frac{a+1}{2}-1}, p x^{2a+\frac{a+1}{2}}, \\
 &  & q x^{a+1}, q x^{a+2}, \dots, q x^{a+\frac{a+1}{2}-1}, q x^{a+\frac{a+1}{2}}],\\[3pt]
 \end{array}
 \]
 \[
\begin{array}{rcl}
\oc_{2,1} & = & [x^{\frac{a+1}{2}-1}, x^{\frac{a+1}{2}-2}, \ldots, x, 1,\\
 & & q x^{a+\frac{a+1}{2}-1}, q x^{a+\frac{a+1}{2}-2}, \ldots, q x^{a+1}, q x^{a},\\
 & & p x^{2a+\frac{a+1}{2}-1}, p x^{2a+\frac{a+1}{2}-2}, \ldots, p x^{2a+1}, p x^{2a}],
 \end{array}
 \]
 \[
\begin{array}{rcl}
c_{2,2} & = & [1,x, q x^{a+\frac{a+1}{2}}],\\[3pt]
c_{2,3} & = &  [px^{2a+\frac{a+1}{2}}, q x^a, q x^{a+1}],\\[3pt]
c_{2,4} & = &  [1, qx^{a+\frac{a+1}{2}}, qx^{a+\frac{a+1}{2}-1}],\\[3pt]
\oc_{2,2} & = &  [q x^a, p x^{2a+\frac{a+1}{2}}, p x^{2a+\frac{a+1}{2}-1}],\\[3pt]
\oc_{2,3} & = &  [x^{\frac{a+1}{2}}, x^{\frac{a+1}{2}-1}, p x^{2a}],\\[3pt]
\oc_{2,4} & = &  [p x^{2a}, p x^{2a+1}, x^\frac{a+1}{2}]
\end{array}
\]
and, finally, consider the following $1$--cells
\[
\begin{array}{rcl}
c_{1,1} & = & [1,x],\\[3pt]
c_{1,2} & = & [1,q x^{a+\frac{a+1}{2}-1}],\\[3pt]
c_{1,3} & = & [1,q x^{a+\frac{a+1}{2}}],\\[3pt]
c_{1,4} & = & [x,q x^{a+\frac{a+1}{2}}].
\end{array}
\]
and let $c_0$ be the vertex $1$. Note that we have the relations
\[
\begin{array}{rcl}
\oc_{2,1} & = & -p x^{2a-1} c_{2,1},\\[3pt]
\oc_{2,2} & = & -p x^{2a+\frac{a+1}{2}-1} c_{2,2},\\[3pt]
\oc_{2,3} & = & -pq  x^{2a+\frac{a+1}{2}-1}c_{2,3},\\[3pt]
\oc_{2,4} & = & -x^\frac{a+1}{2} c_{2,4}.
\end{array}
\]
\begin{center}
\begin{figure}
    \centering
    \def\svgwidth{1.0\columnwidth}
    {\scriptsize
    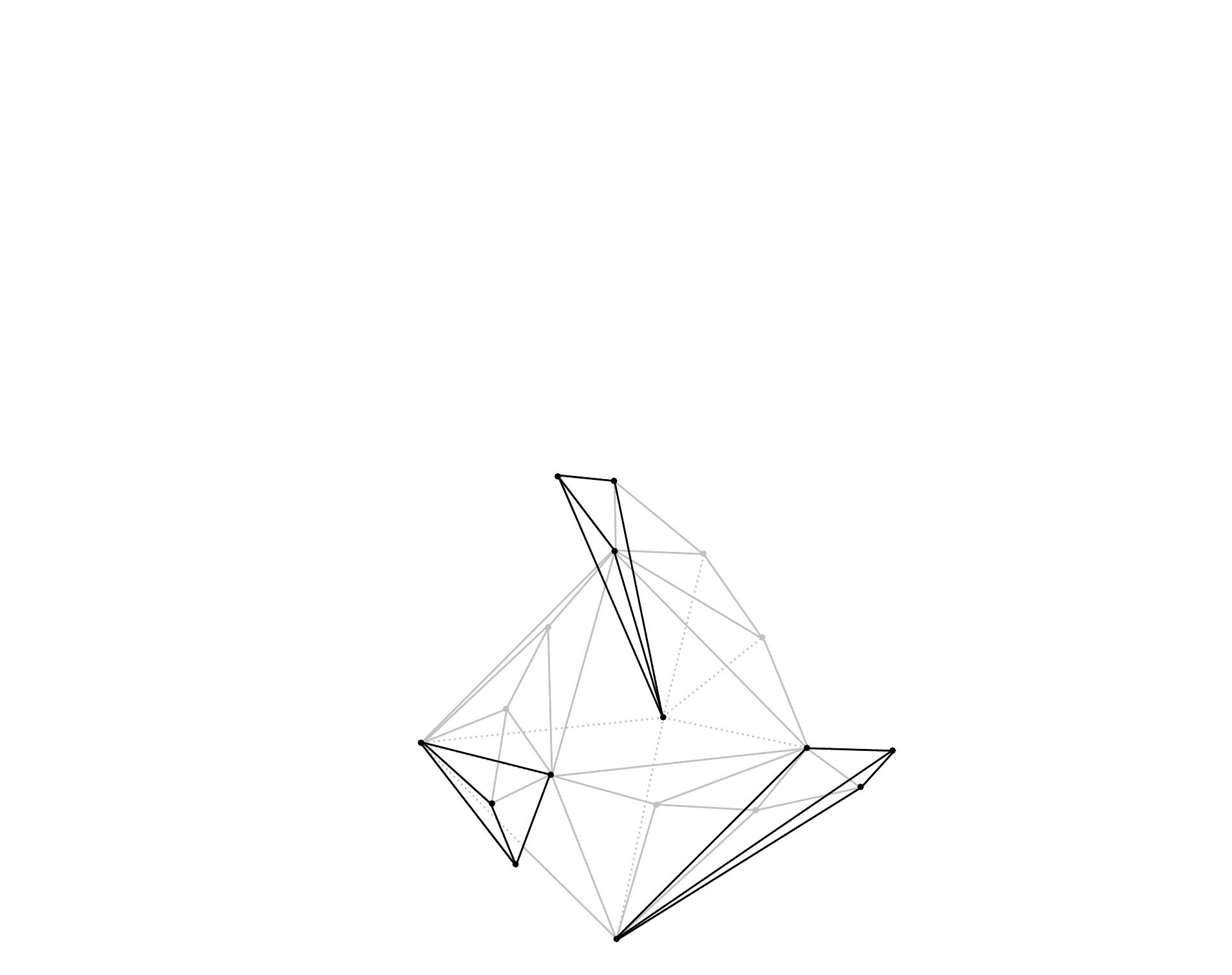
    }
  \caption{The eight $2$--cells (for $s=3$).}\label{figure_cells}
\end{figure}
\end{center}
So we define the following complex of $\Z\Ps$--modules
\[
\xymatrix{
C_\bu:&0\ar[r] & C_3\ar[r]^{\b_3} & C_2\ar[r]^{\b_2} & C_1\ar[r]^{\b_1} & C_0\ar[r] & 0,
}
\]
with
\[
\begin{aligned}
C_3 &= \Z \Ps[c_{3}],\\[3pt]
C_2 &= \Z \Ps[c_{2,1},c_{2,2},c_{2,3},c_{2,4}],\\[3pt]
C_1 &= \Z \Ps[c_{1,1},c_{1,2},c_{1,3},c_{1,4}],\\[3pt]
C_0 &= \Z \Ps[c_0]
\end{aligned}
\]
and, setting $L=\sum_{h=1}^{\frac{a-1}{2}}x^h\in\Z\Ps$, boundary maps given by
\[
\begin{array}{rcl}
\b_3(c_3) & = & c_{2,1}+c_{2,2}+c_{2,3}+c_{2,4}+\\
 & & \oc_{2,1}+\oc_{2,2}+\oc_{2,3}+\oc_{2,4}\\
 & = & (1 - p x^{2a-1}) c_{2,1} + (1 - p x^{2a+\frac{a+1}{2}-1}) c_{2,2} +\\
 & &  (1 - pq  x^{2a+\frac{a+1}{2}-1}) c_{2,3} + (1 - x^\frac{a+1}{2}) c_{2,4},
\end{array}
\]
\[
\begin{array}{rcl}
\b_2(c_{2,1}) & = & (L+px^{2a} L+qx^aL) c_{1,1} -qx^{a+1}c_{1,2}+x^\frac{a+1}{2}c_{1,3}-c_{1,4},\\[3pt]
\b_2(c_{2,2}) & = & c_{1,1}-c_{1,3}+c_{1,4},\\[3pt]
\b_2(c_{2,3}) & = & qx^a c_{1,1}+qx^{a+1} c_{1,2}+px^ {2a+\frac{a+1}{2}-1} c_{1,4},\\[3pt]
\b_2(c_{2,4}) & = & -q x^{a+\frac{a+1}{2}-1}c_{1,1}-c_{1,2}+c_{1,3},
\end{array}
\]
\[
\begin{array}{rcl}
\b_1(c_{1,1}) & = & (x-1)c_0,\\[3pt]
\b_1(c_{1,2}) & = & (qx^{a+\frac{a+1}{2}-1}-1)c_0,\\[3pt]
\b_1(c_{1,3}) & = & (qx^{a+\frac{a+1}{2}}-1)c_0,\\[3pt]
\b_1(c_{1,4}) & = & (qx^{a+\frac{a+1}{2}}-x)c_0.
\end{array}
\]
As in Corollary \ref{corollary_fundamental_generic}, the complex $C_\bu$ may be defined for any free action $V_\ell$, with $1\leq \ell<3^s$ an integer prime to $3$; one just need to replace any occurrence of $x$, $p$ and $q$ with $x_\ell$, $p_\ell$ and $q_\ell$, respectively; also $L$ is replaced by $L_\ell = \sum_{j=1}^{\frac{a-1}{2}}x_\ell^j$; we denote the resulting complex by $C_{\bu,(\ell)}$. Thus we have the following result.

\begin{theorem}\label{theorem_cellular_decomposition_s3} The complex $C_{\bu,(\ell)}$ is a $\Ps$--equivariant cellular chain complex for $\S^3$  with respect to the action $\alpha_\ell$.
\end{theorem}

We denote the cells of the complex $C_{\bu,(\ell)}$ corresponding to $c_3$, $c_{2,j}$, $c_{1,j}$, with $j=1,2,3,4$, and $c_0$ by $c_{3,(\ell)}$, $c_{2,j,(\ell)}$, $c_{1,j,(\ell)}$ and $c_{0,(\ell)}$, respectively. These cells are key for the construction of a complex for the higher dimensional spheres, our next aim.

Indeed, as recalled in Section \ref{section_gen_tetr_group_family}, any free action of $\Ps$ on a sphere $\S^{4n-1}$ is induced by a linear action
\[
\alpha=\alpha_{\ell_0}\oplus\dots\oplus\alpha_{\ell_{n-1}}:\Ps \longrightarrow \gU(2n,\C),
\]
where $1\leq \ell_0,\ell_1,\ldots \ell_{n-1}< 3^s$ are integers prime to $3$ (one may also assume $\ell_0\leq \ell_1\leq\cdots\leq \ell_{n-1}$ up to isomorphism). It is clear that a fundamental domain for $\Ps$ on $\S^{4n-1}$ is given by
\[
c_{4n-1}=\underbrace{\S^3\joint\S^3\joint\cdots\joint\S^3}_{n-1}\joint c_{3,(\ell_{n-1})}.
\]
In order to construct a cellular decomposition of $\S^{4n-1}$ we cascade the various cellular complexes $C_{\bu,(\ell_0)}$, $C_{\bu,(\ell_1)}$, $\ldots$, $C_{\bu,(\ell_{n-1})}$. Indeed the border of $c_{4n-1}$ is clearly $\S^3\joint\cdots\joint\S^3\joint\b_3(c_{3,(\ell_{n-1})})$ a union of $\Ps$--translates of the four $(4n-2)$--cells
\[
c_{4n-2,j}=\underbrace{\S^3\joint\S^3\joint\cdots\joint\S^3}_{n-1}\joint c_{2,j,(\ell_{n-1})},\quad j=1,2,3,4,
\]
and, in the same way, the border of these $(4n-2)$--cells is a union of $\Ps$--translates of the four $(4n-3)$--cells
\[
c_{4n-3,j}=\underbrace{\S^3\joint\S^3\joint\cdots\joint\S^3}_{n-1}\joint c_{1,j,(\ell_{n-1})},\quad j=1,2,3,4.
\]
The border of these last cells is in turn a union of $\Ps$--translates of the the unique $(4n-4)$--cell
\[
c_{4n-4}=\underbrace{\S^3\joint\S^3\joint\cdots\joint\S^3}_{n-1}\joint c_{0,(\ell_{n-1})}.
\]
This cell is a cone over the $(4n-5)$--sphere $\underbrace{\S^3\joint\S^3\joint\cdots\joint\S^3}_{n-1}$ with vertex the point $(0,0,\ldots,0,c_{0,(\ell_{n-1})})$. Hence its border is the $(4n-5)$--sphere which is a union of $\Ps$--translates of the $(4n-5)$--cell
\[
c_{4n-5}=\underbrace{\S^3\joint\S^3\joint\cdots\joint\S^3}_{n-2}\joint c_{3,(\ell_{n-2})}\times0.
\]
Continuing in this way we construct a complex whose description we summarise in the following theorem. In certain boundary maps we need the element $\Sigma=\sum_{g\in\Ps}g$; note that, being $\varphi_\ell^{-1}\circ\varphi_{\oell}$ an automorphism of $\Ps$ for any integer $\ell$ prime to $3$, replacing $x$, $p$ and $q$ by $x_\ell$, $p_\ell$ and $q_\ell$, respectively, does not change $\Sigma$.

\begin{theorem}\label{theorem_higher_dimensional_complex} Let $C_\bu(\S^{4n-1};\Z P'_{8\cdot 3^s},\alpha)$ be the complex of $\Z\Ps$--modules whose generators are the cells $c_{4k+1}$, $c_{4k+2,j}$, $c_{4k-3,j}$ and $c_{4k}$, for $k=0,1,\cdots,n-1$ and $j=1,2,3,4$, as defined above; and whose boundary maps $\b_{4k+1}$, $\b_{4k+2}$, $\b_{4k+3}$ are induced by the boundary of the complex $C_{\bu,(\ell_k)}$ while, for degree $4k$, we have $\b_{4k}(c_{4k}) = \Sigma\cdot c_{4k-1}$ for $k>0$ and $\b_0(c_0)=0$. Then $C_\bu(\S^{4n-1};\Z P'_{8\cdot 3^s},\alpha)$ gives a cellular chain complex for $\S^{4n-1}$ that is equivariant with respect to the action $\alpha_{\ell_0}\oplus\dots\oplus\alpha_{\ell_{n-1}}$ of $\Ps$.
\end{theorem}

The same result is true considering a sequence $\alpha=\alpha_{h_1},\alpha_{h_2},\ldots$ of a denumerable number of free actions; in this case we have a $\Ps$--equivariant cellular decomposition of $\S^{\infty}$. As a consequence we have a resolution of $\Z$ over $\Ps$; however in the next subsection we are going to see an amelioration of this result by using a simpler complex.

\subsection{A complex with lower ranks and resolution} Now we want to define a subcomplex of $C_\bu(\S^{4n-1};\Z P'_{8\cdot 3^s},\alpha)$, chain equivalent to $C_\bu(\S^{4n-1};\Z P'_{8\cdot 3^s},\alpha)$ itself, but that has fewer generators. Due to the $4$--periodicity of the complex we focus on $2$--cells; also we assume that $\ell_k=\oell=1+3^{s-1}$, for any $k=0,\ldots,n-1$, for the moment, then we pass to the general case. The geometric idea is to locate fewer possible $2$--chains by which one can still write the border of the unique $3$--cells. The first attempt should be to take the union of certain $2$--cells; but this does not work directly since in the relations above among cells $\oc_{2,j}$ and cells $c_{2,j}$ all group elements are different.

So we try another approach. Select one cell, for example $c_{2,2}$, and one of its neighbourhoods cells, say $c_{2,4}$, and let $a_1$ be the union of these two cells, namely
\[
a_1=c_{2,2}+c_{2,4}.
\]
Then, $px^{2a+\frac{a+1}{2}-1} a_1=-\bar c_{2,2}+y$, where $y=px^{2a+\frac{a+1}{2}-1}c_{2,4}$ does not belong to the boundary of $c_3$. However, we may find another pair of neighbourhood cells such that one of them is mapped onto $y$ by some group element, while the other one is mapped on some cell in the boundary of $c_3$. For example, take
\[
a_2=c_{2,1} + \bar c_{2,4};
\]
it is easy to see that $px^{2a-1}a_2 = -\bar c_{2,1}-y$. As a consequence
\[
px^{2a+\frac{a+1}{2}-1} a_1+px^{2a-1}a_2 = -\bar c_{2,2}-\bar c_{2,1},
\]
and this means that we can use the three $2$--cells $a_1, a_2$, and $ c_{2,3}$ to cover all the boundary of $c_3$ up to the action of $\Ps$.

We want to lower still the $2$--cells number, so note that if consider the $2$--cell $w = [x,qx^{a+\frac{a+1}{2}}, qx^{a+\frac{a+1}{2}+1}]$, then $px^{2a+\frac{a+1}{2}-1} w = c_{2,3}$ and $px^{2a-1} w = -\bar c_{2,3}$ and these last two group elements are the same appearing in the previous equation. So we define
\begin{align*}
e_{2,1} &= a_1 - w = c_{2,2}+c_{2,4}-pqx^\frac{a+1}{2} c_{2,3},\\
e_{2,2} &= a_2 + w = c_{2,1}-x^\frac{a+1}{2} c_{2,4}+pqx^\frac{a+1}{2} c_{2,3}
\end{align*}
and we finally have
\[
\b_3(c_3)=(1 - px^{2a+\frac{a+1}{2}-1})e_{2,1} + (1 - px^{2a-1}) e_{2,2}.
\]
Moreover the boundary of $e_{2,1}$ and $e_{2,2}$ is completely described using the two $1$--chains $c_{1,1}$ and $c_{1,2}$.

This suggests to consider the following $\Z\Ps$--module complex $E_\bu(\S^{4n-1};\Z P'_{8\cdot 3^s},\alpha)$
\[
\xymatrix{
0\ar[r] & E_{4n-1}\ar[r]^{\b_{4n-1}} & E_{4n-2}\ar[r]^{\b_{4n-2}} & \cdots\ar[r]^{\b_2} & E_1\ar[r]^{\b_1} & E_0\ar[r] & 0
}
\]
with generators: one generator $e_{4k+3}$ in dimension $4k+3$, two generators $e_{4k+2,1}$, $e_{4k+2,2}$ in dimension $4k+2$, two generators $e_{4k+1,1}$, $e_{4k+1,2}$ in dimension $4k+1$ and, finally, one generator $e_{4k}$ in dimension $4k$, and boundaries
\begin{align*}
\b_{4k+3}(e_{4k+3}) &= (1 - p_kx_k^{2a+\frac{a+1}{2}-1})e_{4k+2,1} + (1 - p_kx_k^{2a-1}) e_{4k+2,2},
\end{align*}
\begin{align*}
\b_{4k+2}(e_{4k+2,1}) &= (1 - q_kx_k^{a+\frac{a+1}{2}} - q_kx_k^{a+\frac{a+1}{2}-1})e_{4k+1,1} - (1 + q_kx_k^{a+\frac{a+1}{2} + 1})e_{4k+1,2},\\
\b_{4k+2}(e_{4k+2,2}) &= (L_k + p_kx_k^{2a} L_k + q_kx_k^aL_k + p_kx_k^{2a} + q_k x_k^{a+\frac{a+1}{2}}) e_{4k+1,1}\\
 &  +\quad(x_k^{\frac{a+1}{2}} + q_kx_k^{a+\frac{a+1}{2}+1} - q_kx_k^{a+1})e_{4k+1,2},
\end{align*}
\begin{align*}
\b_{4k+1}(e_{4k+1,1})&=(x_k-1)e_{4k},\\
\b_{4k+1}(e_{4k+1,2})&=(q_kx_k^{a+\frac{a+1}{2}-1} - 1)e_{4k},
\end{align*}
\begin{align*}
\b_{4k}(e_{4k}) &= \Sigma\cdot e_{4k - 3},\textrm{ if }k>0,
\end{align*}
where we set $x_k=x_{\ell_k}, p_k=p_{\ell_k}$ and $q_k=q_{\ell_k}$ for $k=0,\ldots,n-1$.

\begin{pro}\label{proposition_chainEquivalence} The complexes $C_\bu(\S^{4n-1};\Z P'_{8\cdot 3^s},\alpha)$ and $E_\bu(\S^{4n-1};\Z P'_{8\cdot 3^s},\alpha)$ are chain equivalent.
\end{pro}
\begin{proof} For short, let us write $C_\bu$ and $E_\bu$ for the two complexes. Consider the $\Z\Ps$--map $\va'_\bu :C_\bu\to E_\bu$ defined by
\begin{align*}
\va'_{4k+3}(c_{4k+3}) &= e_{4k+3},\\
\va'_{4k+2}(c_{4k+2,1}) &= e_{4k+2,2}, & \va'_2(c_{4k+2,2}) &= e_{4k+2,1},\\
\va'_{4k+2}(c_{4k+2,3}) &= 0, & \va_2'(c_{4k+2,4}) &= 0,\\
\va'_{4k+1}(c_{4k+1,1}) &= e_{4k+1,1}, & \va'_1(c_{4k+1,2}) &= e_{4k+1,2},\\
\va'_{4k+1}(c_{4k+1,3}) &= qx^{a+\frac{a+1}{2}-1}e_{4k+1,1} + e_{4k+1,2},\\
\va_1'(c_{4k+1,4}) &= -qx^{a+\frac{a+1}{2}}e_{4k+1,1}-qx^{a+\frac{a+1}{2}+1}e_{4k+1,2},\\
\va'_{4k}(c_{4k}) &= e_{4k};
\end{align*}
and consider also the $\Z\Ps$--map $\va_\bu :E_\bu\to C_\bu$ defined by
\begin{align*}
\va_{4k+3}(e_{4k+3}) & = c_{4k+3},\\
\va_{4k+2}(e_{4k+2,1}) &= c_{4k+2,2} - p_kq_kx_k^\frac{a+1}{2}c_{4k+2,3} + c_{4k+2,4},\\
\va_2(e_{4k+2,2}) &= c_{4k+2,1} + p_kq_kx_k^\frac{a+1}{2}c_{4k+2,3} - x_k^\frac{a+1}{2}c_{4k+2,4},\\
\va_{4k+1}(e_{4k+1,1}) &= c_{4k+1,1},\\
\va_1(e_{4k+1,2}) &= c_{4k+1,2},\\
\va_{4k}(e_{4k}) &= c_{4k}.
\end{align*}
It is routine to check that $\va'$ and $\va$ are chain maps; moreover one easily proves that $\va'_\bu\circ\va_\bu = \Id_{E_\bu}$. On the other hand the collection of $\Z\Ps$--module maps $D_\bu:C_\bu\longrightarrow C_{\bu}$
\begin{align*}
D_{4k+3} &= 0,\\
D_{4k+2} &= 0,\\
D_{4k+1}(c_{4k+1,1}) &=0, & D_{4k+1}(c_{4k+1,2}) &= 0,\\
D_{4k+1}(c_{4k+1,3}) &= -c_{4k+2,4}, & D_{4k+1}(c_{4k+1,4}) &= -p_kq_kx_k^\frac{a+1}{2}c_{4k+2,3},\\
D_{4k} &= 0
\end{align*}
is a chain homotopy from $\va_\bu\circ\va'_\bu$ to the identity of $C_\bu$.
\end{proof}

Since $C_\bu$ is acyclic, also $E_\bu$ is acyclic, hence, considering a sequence $\alpha=\alpha_{\ell_0},\alpha_{\ell_1},\ldots$ of a denumerable number of free actions, we have the following result.
\begin{cor}\label{theorem_resolution_C} The augmentation of the complex $E_\bu(\S^\infty;\Ps,\alpha)$ is a resolution of $\Z$ over $\Ps$. In particular if we set $\ell_1=\ell_2=\ell_3=\cdots$, we have a $4$--periodic resolution of $\Z$ over $\Ps$.
\end{cor}

Now we show that the above constructed resolutions have minimal ranks. We need some notations, following the paper \cite{swa} of Swan. Let $G$ be a group and let 
\[
\xymatrix{
F_\bu: & \cdots \ar[r]&F_2\ar[r] & F_1\ar[r] & F_0\ar[r]^{\epsilon}&\Z\ar[r]&0
}
\]
be a resolution of $\Z$ over $\Z G$, let $f_h(F_\bu)=\rank_{\Z G}F_h$, $h\geq0$, and, assuming that all the ranks are finite, define $\mu_h(F_\bu)=f_h-f_{h-1}+f_{h-2}-\cdots + (-1)^hf_0$. Further let $\mu_h(G)$ be the infimum of $\mu_h(F_\bu)$ over all such resolution $F_\bu$.
\begin{pro}\label{proposition_Q8_subgroup}
If $G$ is a finite group having $\Qua$ as a subgroup, then for all $h\geq0$ we have $\mu_h(G)\geq1$ for $h\equiv 0,3\pmod{4}$ and $\mu_h(G)\geq2$ otherwise.
\end{pro}
\begin{proof} First we want to compute the values of $\mu_h(\Qua)$, for $h\geq0$. We begin by recalling the integral homology groups (see, for example \cite{MMS})
\[
H_h(\Qua;\Z) = \left\{
\begin{array}{ll}
\Z & \textrm{if } h = 0,\\
\Z/2\Z\oplus\Z/2\Z & \textrm{se } h\equiv1\pmod{4}\\
\Z/8\Z & \textrm{if } h\equiv3\pmod{4}\\
0 & \textrm{if } n>0,\,n \textrm{ even}.
\end{array}
\right.
\]
Denoting by $s_h$ the minimal number of generators of $H_h(\Qua;\Z)$ and by $b_h$ the dimension over $\Q$ of $H_h(\Qua;\Z)\otimes\Q$, we have
\[
\mu_h(\Qua)\geq s_h - b_{h-1} + b_{h-2} - \cdots + (-1)^hb_0
\]
by Theorem~1.1 of \cite{swa}. We find $\mu_h(\Qua)\geq 0$ if $h\equiv 3\pmod{4}$ and $\mu_h(\Qua)\geq 1$ otherwise. But, being $\Qua$ a $2$--group, by \cite[pag. 193]{swa} there exists a free resolution $F_\bu$ of $\Z$ over $\Z\Qua$ with $f_h(F_\bu)=\dim_{\F_2}H^h(G;\F_2)$; hence (using for example \cite[pag.~129]{AM} for these cohomology groups), we have $f_h(F_\bu)=1$ for $h\equiv0,3\pmod{4}$ while $f_h(F_\bu)=2$ for $h\equiv1,2\pmod{4}$. This resolution and the above inequalities show that $\mu_h(\Qua) = 0$ if $h\equiv 3\pmod{4}$ and $\mu_h(\Qua) = 1$ otherwise.

Now, by Theorem~2.1 of \cite{swa}, we find $[G:\Qua]\cdot\mu_h(G)\geq\mu_h(\Qua)$, for all $h\geq0$; hence $\mu_h(G)\geq0$ if $h\equiv3\pmod{4}$ and $\mu_h(G)\geq 1$ otherwise. Finally let $F_\bu$ be a free resolution of $\Z$ over $\Z G$, then
\[
f_h(F_\bu) = \mu_h(F_\bu)+\mu_{h-1}(F_\bu)\geq \mu_h(G) + \mu_{h-1}(G)
\]
and our claim follows by the inequalities for $\mu_h(G)$.
\end{proof}

\begin{cor}\label{corollary_minimal_ranks} The augmentation of the complex $E_\bu(\S^\infty;\Ps,\alpha)$ has minimal ranks as a resolution of $\Z$ over $\Z\Ps$.
\end{cor}

\subsection{Homology and cohomology groups}

The computation of the homology and cohomology groups of $\Ps$ follows easily using the complex $E_\bu(\S^\infty;\Z P'_{8\cdot 3^s},\alpha)$. In the following theorem we see the homology groups with integer and mod $3$ coefficients and trivial action. Note that $E_\bu(\S^\infty;\Z P'_{8\cdot 3^s},\alpha)\otimes_{\Z\Ps} A$ is independent of $\alpha$ as long as $\Ps$ acts trivially on the ring $A$, so we denote this new complex simply by $E_\bu\otimes_{\Z\Ps} A$. In the following theorems we identify cycles and cocycles with their classes in homology and cohomology, respectively.
\begin{theorem}\label{theomem_homology_groups}
The homology groups of $\Ps$ with $\Z$ coefficients are given by
\begin{align*}
H_0(\Ps;\Z) &=\langle e_0\rangle\simeq\Z,\\
H_{4k}(\Ps;\Z) &= 0\textrm{ if }k>0,\\
H_{4k+1}(\Ps;\Z) &= \langle e_{4k+1,1}\rangle = \langle e_{4k+1,2}\rangle \simeq \Z / 3^s\Z,\\
H_{4k+2}(\Ps;\Z) &= 0,\\
H_{4k+3}(\Ps;\Z) &= \langle e_{4k+3}\rangle \simeq \Z / 8\cdot 3^s\Z;
\end{align*}
while those with $\Z/3\Z$ coefficients are
\begin{align*}
H_{4k}(\Ps;\Z/3\Z) &= \langle e_{4k} \rangle \simeq \Z/3\Z,\\
H_{4k+1}(\Ps;\Z/3\Z) &= \langle e_{4k+1,1} \rangle = \langle e_{4k+1,2} \rangle \simeq \Z/3\Z,\\
H_{4k+2}(\Ps;\Z/3\Z) &= \langle e_{4k+2,1} - e_{4k+2,2} \rangle \simeq \Z/3\Z,\\
H_{4k+3}(\Ps;\Z/3\Z) &= \langle e_{4k+3} \rangle \simeq \Z/3\Z.
\end{align*}
\end{theorem}
\begin{proof} The computations are straightforward, as an example we give those for $\Z$. By $4$--periodicity it suffices to consider only the low degrees. Note that $|L_k|=(a-1)/2=(3^{s-1} - 1)/2$, so, denoting by $N$ the integer $3|L_k|+2 = (3^s+1)/2$, the boundaries of $E_\bu\otimes_{\Z[\Ps]}\Z$ are
\begin{align*}
\b_4(e_4) &= 8\cdot 3^s e_3,\\
\b_3(e_3) &= 0,\\
\b_2(e_{2,1}) &= -e_{1,1} - 2e_{1,2},\\
\b_2(e_{2,2}) &= Ne_{1,1} + e_{1,2},\\
\b_1(e_{1,1}) &= 0,\\
\b_1(e_{1,2}) &= 0,\\
\b_0(e_0) &= 0.
\end{align*}
All homology groups follows at once but that in degree $1$. Making the bases changes $\bar e_{2,1} = -N e_{2,1}-e_{2,2}$, $\bar e_{2,2}= -e_{2,1}$ and $\bar e_{1,1} = e_{1,2}$, $\bar e_{1,2} = e_{1,1} + 2e_{1,2}$ we have $\b_2(\bar e_{2,1}) = 3^s\bar e_{1,1}$, $\b_2(\bar e_{2,2}) = \bar e_{1,2}$. The $H_1$ is now clear.
\end{proof}
The computations of the cohomology groups of $\Ps$ are very similar and we just report the results for integer and mod $3$ coefficients where we use the dual of the generators of $E_\bu\otimes_{\Z\Ps}\Z$ and of $E_\bu\otimes_{\Z\Ps}\Z/3\Z$, respectively.
\begin{theorem}\label{theorem_cohomology_groups}
The cohomology groups of $\Ps$ with $\Z$ coefficients are given by
\begin{align*}
H^{0}(\Ps;\Z) &= \langle e^0\rangle \simeq \Z,\\
H^{4k}(\Ps;\Z) &= \langle e^{4k}\rangle \simeq \Z / 8\cdot 3^s\Z\textrm{ for }k>0,\\
H^{4k+1}(\Ps;\Z) &= 0,\\
H^{4k+2}(\Ps;\Z) &= \langle e^{4k+2}_1\rangle = \langle e^{4k+2}_2\rangle \simeq \Z / 3^s\Z,\\
H^{4k+3}(\Ps;\Z) &= 0;
\end{align*}
and those with $\Z/3\Z$ coefficients are
\begin{align*}
H^{4k}(\Ps,\Z/3\Z) &= \langle e^{4k}\rangle \simeq \Z/3\Z,\\
H^{4k+1}(\Ps,\Z/3\Z) &= \langle e^{4k+1}_1 + e^{4k+1}_2\rangle \simeq \Z/3\Z,\\
H^{4k+2}(\Ps,\Z/3\Z) &= \langle e^{4k+2}_1\rangle = \langle e^{4k+2}_2\rangle \simeq \Z/3\Z,\\
H^{4k+3}(\Ps,\Z/3\Z) &= \langle e^{4k+3}\rangle \simeq \Z/3\Z.
\end{align*}
\end{theorem}

\subsection{The cup product}

We want now to describe the cup product in the cohomology of the group $\Ps$ using our resolution $E_\bu(\S^\infty;\Z\Ps,\alpha)$. Our result is the following
\begin{theorem}\label{theorem_cohomology_ring} The cohomology ring $H^\bu(\Ps;\Z)$ is isomorphic to the graded (commutative) polynomial ring quotient
\[
\Z[x,y]\,\Big/\left(3^s x = 8\cdot 3^s y = 0,\,x^2 = 8y\right),
\]
where $x$ has degree $2$ and $y$ has degree $4$, by an isomorphism mapping $x$ to the class of $e_1^2$ and $y$ to that of $e^4$.
\end{theorem}
\begin{proof}
First of all, as in the computation of the homology and cohomology groups, the representation $\alpha$ does not influence at all the cup product since we consider the complex $E_\bu\otimes_{\Z\Ps} \Z = E_\bu(\S^\infty;\Z P'_{8\cdot 3^s},\alpha)\otimes_{\Z\Ps} \Z$ with the trivial action of $\Ps$ on $\Z$. So in what follows we set $\alpha=\alpha_{\oell}\oplus\alpha_{\oell}\oplus\cdots$.

Let $Z=\langle z\rangle$ be the cyclic subgroup of $\Ps $of order $3^s$ generated by $z$ and consider the following free resolution $\calZ_\bu$ of $\Z$ as a trivial $Z$--module: in degree $k$, $k \geq 0$, the $\Z Z$--free module of rank one $Z_k = \Z Z[\tc_k]$ generated by $\tc_k$
\[
\xymatrix{
	\cdots \ar[r] & Z_4\ar[r]^{\b_4} & Z_3\ar[r]^{\b_3} & Z_2\ar[r]^{\b_2} & Z_1\ar[r]^{\b_1} & Z_0\ar[r]^\epsilon & \Z
}
\]
where, setting $\Theta = \sum_{h=0}^{3^s-1}z^h$, the boundaries are defined by
\begin{align*}
\b_{4k+4}(\tc_{4k+4}) &= \Theta\tc_{4k+3},\\
\b_{4k+3}(\tc_{4k+3}) &= (z^{2a+\frac{a+1}{2}-1} - 1)\tc_{4k+2},\\
\b_{4k+2}(\tc_{4k+2}) &= \Theta\tc_{4k+1},\\
\b_{4k+1}(\tc_{4k+1}) &= (z-1)\tc_{0}.
\end{align*}
As for any resolution of a cyclic group, we know that the ring structure on the cohomology of $\calZ\otimes_{\Z Z}\Z$ is isomorphic to the polynomial ring quotient $\Z[\tc^2]/(3^s\cdot\tc^2=0)$ with $\tc^2$ in degree $2$.

Now consider the category $\calC$ whose objects are the pairs $(G,\calR)$ where $G$ is a group and $\calR$ a chain complex of $\Z G$--modules , and whose morphisms are the pairs $(f,\gamma):(G,\calR)\longrightarrow(H,\calR')$ with $f$ a group homomorphism $f:G\longrightarrow H$ and $\gamma$ a chain map from $\calR$ to $\calR'$ such that $\gamma(g\cdot c)=f(g)\gamma(c)$ for all $g\in G$ and $c\in\calR$.

Clearly $(\Ps,E_\bu)$ and $(Z,\calZ_\bu)$ are objects of $\calC$ and taking: as $f:\Ps\longrightarrow Z$ the group homomorphism induced by $z\longmapsto z$, $p,q\longmapsto 1$, and as chain map the one induced by
\begin{align*}
\gamma_{-1}(1) &= 1,\\
\gamma_{4k}(e_{4k}) &= 2^{3k}\tc_{4k},\\
\gamma_{4k+1}(e_{4k+1,1}) &= 2^{3k} \tc_{4k+1},\\
\gamma_{4k+1}(e_{4k+1,2}) &= 2^{3k} (1 + z + z^2 + \cdots + z^{2a+\frac{a+1}{2}-2})\tc_{4k+1},\\
\gamma_{4k+2}(e_{4k+2,1}) &= -2^{3k} \tc_{4k+2},\\
\gamma_{4k+2}(e_{4k+2,2}) &= 2^{3k} \tc_{4k+2},\\
\gamma_{4k+3}(e_{4k+3}) &= -2^{3k} z^{2a+\frac{a+1}{2}-1}\tc_{4k+3},\\
\end{align*}
we define a morphism from $(\Ps,E_\bu)$ to $(Z,\calZ_\bu)$ in $\calC$.

The map $\gamma^*$ induced by $\gamma$ in cohomology is a ring homomorphism from $H^\bu(\Ps;\Z)$ to $H^\bu(Z;\Z)$ since $\gamma$ extends the identity map on $\Z$. Now let $\ox = e^2_1$ and $\oy = e^4$. We can easily derive that $\gamma^*(\tc_2)=-\ox$, $\gamma^*(\tc_4)=8\oy$, thus $x^2=8y$ using that $\gamma^*$ is a ring homomorphism.

Note further that the multiplication by $\oy$ from $H^d(\Ps;\Z)$ to $H^{d+4}(\Ps;\Z)$ is an isomorphism for any $d\geq0$ since the cohomology is $4$--periodic (see for example \cite{Bro}). In particular $\ox$ and $\oy$ generate $H^\bu(\Ps;\Z)$ as a ring, with $H^{4k}(\Ps;\Z)$ generated by $\oy^k$ and $H^{4k+2}(\Ps;\Z)$ generated by $\ox\oy^k$ as $\Z$--modules.

At this point we know that the following relations $3^s\ox=8\cdot 3^s\oy=0$ and $\ox^2=8\oy$ hold in $H^\bu(\Ps;\Z)$. So there exists a surjective graded ring homomorphism from the (commutative) polynomial ring $\Z[x,y]$, with $x$ in degree $2$ and $y$ in degree $4$, to $H^\bu(\Ps;\Z)$ induced by $x\longmapsto\ox$, $y\longmapsto\oy$. Let $R$ be its kernel and let $I$ be the ideal generated by $3^sx$, $8\cdot 3^sy$ and $x^2-8y$; we know that $I\subseteq R$. We want to show that $R = I$; this will clearly finish our proof.

The component of degree $4k$ in $\Z[x,y]/I$ is generated by $y^k$ as a $\Z$--module and it is a submodule of $\Z/8\cdot 3^s\Z$ if $k>0$; similarly the component of degree $4k+2$ is generated by $x y^k$ and is a submodule of $\Z/3^s\Z$. But we have the surjective homogeneous quotient maps $\Z[x,y]/I \longrightarrow \Z[x,y]/R \longrightarrow H^\bu(\Ps;\Z)$ and, as we have just proved, the first ring has components that are isomorphic to submodules of the last rings. This shows that the two maps are isomorphisms and, in particular, $R = I$.
\end{proof}

\subsection{Reidemeister Torsion}

Now we compute the Reidemeister torsion $\tau(X_\alpha)$ of a \emph{tetrahedral space form} $X_\alpha=\S^{4n-1}/\Ps$, where $\Ps$ acts via the free action $\alpha=\alpha_{\ell_0}\oplus\cdots\oplus\alpha_{\ell_{n-1}}$ for some integers $1\leq \ell_0,\ldots,\ell_{n-1} < 3^s$ prime to $3$.

Recall that we have constructed the $\Z\Ps$--module complexes $C_\bu$ and $E_\bu$ related to $\alpha$. The first step is to consider the complexes $V_\bu=C_\bu\otimes_{\Z\Ps}\C$ and $U_\bu=E_\bu\otimes_{\Z\Ps}\C$ via the representation $\Ps\longrightarrow\C^*$ defined by $p\longmapsto 1$, $q\longmapsto 1$ and $z\longmapsto\zeta$, where $\zeta=e^{2\pi/3^s}$. Our aim is the computation of the Reidemeister torsion $\tau(X_\alpha)=\tau(V_\bu)$ as an element of $\C^*/\Gamma$ where $\Gamma$ is the subgroup generated by $-\zeta$. But we show that $\tau(V_\bu)$ and $\tau(U_\bu)$ are equal, hence we may use the simpler complex $U_\bu$.

The complex $C_\bu$ has a preferred $\Z\Ps$--basis as in its definition, so $V_\bu$ has a preferred $\C$--basis; in the same way also $U_\bu$ has a preferred $\C$--basis. Moreover $V_\bu$ and $U_\bu$ are chain equivalent via the map $\psi_\bu=\varphi_\bu\otimes_{\Z\Ps}\Id_\C$ and $\psi'_\bu=\varphi'_\bu\otimes_{\Z\Ps}\Id_\C$, where $\varphi$ and $\varphi'$ are defined in the proof of Proposition \ref{proposition_chainEquivalence}. Moreover $\psi_\bu:U_\bu\longrightarrow V_\bu$ is an injective chain map, so we may consider the exact sequence of $\C$--vector space complexes
\[
0\longrightarrow U_\bu \longrightarrow V_\bu \longrightarrow W_\bu\to 0,
\]
where $W_\bu$ is the quotient complex with preferred basis defined as follows
\begin{align*}
W_{4k}   &= 0 + V_{4k},\\
W_{4k+1} &= \langle c_{4k+1,3} + V_{4k+1}, c_{4k+1,4} + V_{4k+1}\rangle_\C,\\
W_{4k+2} &= \langle c_{4k+2,3} + V_{4k+2}, c_{4k+2,4} + V_{4k+2}\rangle_\C,\\
W_{4k+3} &= 0 + V_{4k+3}
\end{align*}
for any $0\leq k\leq n-1$, and unique non trivial boundary
\begin{align*}
\b_{4k+2}(c_{4k+2,3} + V_{4k+2}) &= \zeta_{h_k}^ {2a+\frac{a+1}{2}-1} c_{4k+1,4} + V_{4k+1},\\
\b_{4k+2}(c_{4k+2,4} + V_{4k+2}) &= c_{4k+1,3} + V_{4k+1}
\end{align*}
for any $0\leq k\leq n-1$.

The complexes $U_\bu$ and $W_\bu$ are acyclic by easy direct verification, so also $V_\bu$ is acyclic since it is chain equivalent to $U_\bu$. We consider a new basis of $V_\bu$ defined as follows
\begin{itemize}
\item[-] in degree $4k$: $\varphi_{4k}(e_{4k})$,
\item[-] in degree $4k+1$: $\varphi_{4k+1}(e_{4k+1,1}),\varphi_{4k+1}(e_{4k+1,2}),c_{4k+1,3},c_{4k+1,4}$,
\item[-] in degree $4k+2$: $\varphi_{4k+2}(e_{4k+2,1}),\varphi_{4k+2}(e_{4k+2,2}),c_{4k+2,3},c_{4k+2,4}$,
\item[-] in degree $4k+3$: $\varphi_{4k+3}(e_{4k+3})$.
\end{itemize}
By construction the preferred basis of $U_\bu$, this new basis of $V_\bu$ and the preferred basis of $W_\bu$ are compatible. But the same is true also for the preferred basis of $V_\bu$ since the matrices of the basis changes to the new basis of $V_\bu$ are all upper triangular with $1$ on the diagonal.

So, by Theorem~3.1 in \cite{Mil2}, we have $\tau(V_\bu)=\tau(U_\bu)\cdot\tau(W_\bu)$ where all torsions are with respect to the preferred basis.

Recall that, given a vector space complex $A_\bu$ and preferred basis $a_k=(a_{k,1},\ldots,a_{k,n_k})$ in degree $k$, for any $0\leq k\leq n-1$, if the complex is acyclic, we may choose vectors $b_k=(b_{k,1},\ldots,b_{k,m_k})$ in $A_k$, for any $0\leq k\leq n$, such that:
\begin{itemize}
  \item[i)] $b_0$ is empty,
  \item[ii)] the boundary operator $\b_k$ restricted to the subspace generated by $b_{k,1},\ldots,b_{k,m_k}$ is injective,
  \item[iii)] the vectors $\b_k(b_{k+1,1}),\ldots,\b_{k+1}(b_{k+1,m_{k+1}}),b_{k,1},\ldots,b_{k,m_k}$ form a basis of $A_k$.
\end{itemize}
In this way, denoting by $(\b_{k+1}(b_{k+1}),b_k / a_k)$ the matrix giving the change of basis from the basis $a_k$ to the basis $\b_{k+1}(b_{k+1}),b_k$, the torsion of $A_\bu$ with respect to the preferred basis $a_k$ is
\[
\tau(A_\bu) = \prod_{k=0}^{n-1}\left(\det(\b_{k+1}(b_{k+1}),b_k / a_k)\right)^{(-1)^k}.
\]
Note that this is independent of the $b_k$'s.

In the following computation we set $\zeta_\ell = \zeta^{r(3^{s-1}+1)}$, where $\ell$ is an integer prime to $3$ and $r$ is such that $r\ell\equiv 1\pmod{3^s}$; note that $x_\ell\longmapsto \zeta_\ell$ in the representation $\Ps\longrightarrow\C^*$ defined above.

By applying the previous formula for the torsion to the complex $W_\bu$ with $b_{4k+2,j} = c_{4k+2,j} + V_{4k+2}$, $j=1,2$, we find
\[
\tau(W_\bu)=\prod_{k=0}^{n-1}\det\left(\begin{array}{cc}0 & \zeta_{\ell_k}^{2a+\frac{a+1}{2}-1}\\1 & 0\end{array}\right)^{-1}=1\in\C^*/\Gamma,
\]
and so $\tau(V_\bu)=\tau(U_\bu)$ as claimed. We now compute the torsion of $U_\bu$.

\noindent In degree $4k+1$ we may chose $b_{4k+1}=e_{4k+1,1}$, hence
\[
\det (\b_{4k+1}(b_{4k+1}) b_{4k}/e_{4k})=\det ((\zeta_{\ell_k}-1)e_{4k}/e_{4k})=\zeta_{\ell_k}-1.
\]

\noindent In degree $4k+2$ we may chose $b_{4k+2}=e_{4k+2,1}$, hence
\begin{align*}
\det (\b_{4k+2}(b_{4k+2}) &b_{4k+1}/e_{4k+1})\\
&=\det \left(
\begin{matrix}
1-\zeta_{\ell_k}^{a+\frac{a+1}{2}}-\zeta_{\ell_k}^{a+\frac{a+1}{2}-1} & 1\\
-1-\zeta_{\ell_k}^{a+\frac{a+1}{2}+1} & 0
\end{matrix}
\right)\\
& = 1+\zeta_{\ell_k}^{a+\frac{a+1}{2}+1}.
\end{align*}

\noindent In degree $4k+3$ we may chose $b_{4k+3}=e_{4k+3}$, hence
\begin{align*}
\det (\b_{4k+3}(b_{4k+3}) b_{4k+2}/e_{4k+2,1} e_{4k+2,2})&=
\det \left(
\begin{matrix}
1 - \zeta_{\ell_k}^{2a+\frac{a+1}{2}-1} & 1\\
1 - \zeta_{\ell_k}^{2a-1} & 0
\end{matrix}
\right)\\
&= \zeta_{\ell_k}^{2a-1}-1,
\end{align*}
and also
\begin{align*}
\det ( b_{4k+3}/e_{4k+3} )&=1.
\end{align*}
We have thus proved the following result.
\begin{theorem}
The tetraedral space form $X_\alpha = S^{4n-1}/\Ps$, where $\alpha=\alpha_{\ell_0}\oplus\cdots\oplus\alpha_{\ell_{n-1}}$, has Reidemeister torsion
\[
\tau(X_\alpha) = \prod_{k=0}^{n-1} \frac{(\zeta_{\ell_k} - 1)(\zeta_{\ell_k}^{2a-1} - 1)}{\zeta_{\ell_k}^{a+\frac{a+1}{2}+1} + 1}\in\C^*/\Gamma.
\]
\end{theorem}

Now we want to show that the factors of the previous formula are multiplicatively independent if the $\ell_k$'s are suitably restricted. In the proof of such property we need the following lemma on circulant matrix. Recall that a square matrix $A=(a_{i,j})$ of order $m$ is \emph{circulant} if $a_{i+1,j+1}=a_{i,j}$ for any $0\leq i,j\leq m-1$ considered modulo $m$; we call the sum $\sum_{i=0}^{m-1} a_{i,0}$ of the first column the \emph{content} of $A$.
\begin{lem} Let $A$ be a circulant matrix with integer coefficients of order a power of an odd prime $p$. If the content of $A$ is not congruent to $0$ modulo $p$, then $A$ has maximal rank.
\end{lem}
\begin{proof} Let $m=p^n$ be the order of $A$ and let $f(x)=\sum_{j=0}^{m-1} a_{j,0}x^j$ be the polynomial associated to $A$. The rank of $A$ is $m-d$ where $d$ is the degree of $\gcd(f(x),x^m-1)$ (see for example \cite{Ing}); so we have to show that $f(x)$ and $x^m-1$ are relative prime. Denoting by $\Phi_d(x)$ the $d$--th cyclotomic polynomial, we have
\[
x^m - 1 = \prod_{k=0}^n \Phi_{p^k}(x).
\]
Now note that $\Phi_{p^k}(x)=\sum_{j=0}^{p-1}x^{j\cdot p^{k-1}}$ if $k>0$, while $\Phi_1(x) = x-1$. In any case $\Phi_{p^k}(1)\equiv 0\pmod{p}$. So if for some $0\leq k\leq n$ the polynomial $\Phi_{p^k}(x)$ divides $f(x)$ in $\Z[x]$, then $f(1)\equiv 0\pmod{p}$. But, by our hypothesis on the content $f(1)$ of $A$ we have $f(1)\not\equiv0\pmod{p}$. This finishes our proof since the cyclotomic polymonials are irreducible in $\Z[x]$.
\end{proof}

\begin{pro} If we define
\[
\tau_\ell = \frac{(\zeta_\ell - 1)(\zeta_\ell^{2a-1} - 1)}{\zeta_\ell^{a+\frac{a+1}{2}+1} + 1}\in\C^*/\Gamma,
\]
then we have
\begin{itemize}
\item[(i)] for any $\ell$, $\tau_\ell=\tau_{-\ell}$ as elements of $\C^*/\Gamma$,
\item[(ii)] the elements $\tau_\ell$, where $\ell$ varies in a set of representatives of $(\Z/3^s\Z)^*$ modulo the subgroup generated by $-1$, are multiplicative independent in $\C^*/\Gamma$.
\end{itemize}
\end{pro}
\begin{proof} The equality in (i) is clear, so we prove the multiplicative independence of (ii). First of all let $u = a + 1$, $v = a - 1$ and $w = (a+1)/2$; these are three integers prime to $3$. If $r$ is such that $r\ell\equiv 1\pmod{3^s}$ then
\[
\delta_r=\tau_\ell=\frac{(\zeta_\ell - 1)(\zeta_\ell^{2a-1} - 1)}{\zeta_\ell^{a+\frac{a+1}{2}+1} + 1} = \frac{(\zeta^{ru} - 1)(\zeta^{rv}-1)}{\zeta^{3rw} + 1},
\]
and, if we define $\ep_d=(\zeta^d-1)/(\zeta-1)$, $d\in\Z$, we have
\[
\delta_r=(\zeta - 1)^2\ep_{ru}\ep_{rv}\frac{\ep_{rw}}{\ep_{2rw}}\frac{\ep_{r(w+a)}}{\ep_{2r(w+a)}}\frac{\ep_{r(w+2a)}}{\ep_{2r(w+2a)}}.
\]
Clearly our claim about the $\tau_\ell$'s is equivalent to the same claim about the $\delta_r$'s.

Let $\calO=\Z[\zeta]$ be the ring of integers of the cyclotomic field $\Q(\zeta)$, let $\calO^*$ be the set of units and let $N:\Q(\zeta)\longrightarrow\Q$ be the norm map. Note that the norm map pass to the quotient $\Q(\zeta)^*/\Gamma$ since $-\zeta$ has norm $1$. Also, any element $\ep_d$, with $d$ prime to $3$, is a unit (it is called a \emph{cyclotomic unit}) in $\calO$ and has norm $1$; instead $N(\zeta-1)=3$ as one can prove at once by noting that the minimal polynomial of $\zeta-1$ is $\Phi_{3^s}(x+1)$, where $\Phi_{3^s}(x)$ is the $3^s$--th cyclotomic polynomial.

Since $\ep_{-d}=\ep_d$ and $\delta_{-r}=\delta_r$ in $\Q(\zeta)^*/\Gamma$, we consider the quotient group $G=(\Z/3^s\Z)^*/\{\pm1\}$. For any multiplicative relation
\[
\prod_{r\in G}\delta_r^{e_r} = 1\in\C^*/\Gamma,
\]
where $e_r$ are integers, we have
\[
\prod_{r\in G}N(\delta_r)^{e_r} = 9^{\sum_{r\in G}e_r} = 1.
\]
Thus any relation is homogeneous: $\sum_{r\in G}e_r = 0$. Moreover any homogeneous relation in the $\delta_r$'s may be written in terms of the $\odelta_r=\delta_r/(\zeta-1)^2$'s. Using the above given expression on $\delta_r$ in terms of the $\ep_d$'s, we see that
\[
\odelta_r = \prod_{q\in G}\ep_{rq}^{a_q}
\]
for certain integers $a_q$, $q\in G$, which does \emph{not} depend on $r$; this is a key point for our proof.

Now note that $G$ is a cyclic group of order $3^{s-1}$ since $(\Z/3^s\Z)^*$ is cyclic; we use this to change the indexing of the $\odelta_r$'s and of the $\ep_d$'s so that the matrix of the $a_q$'s become circulant. Indeed, let $t$ be a fixed generator for $G$ and define $\tdelta_j = \odelta_{t^j}$, $\tep_j=\ep_{t^j}$, $\ta_j=a_{t^j}$; with these definitions we have
\[
\tdelta_j = \prod_{i=0}^{3^{s-1}-1}\ep_{t^i\cdot t^j}^{a_{t^i}}=\prod_{i=0}^{3^{s-1}-1}\tep_{i+j}^{\,\,\ta_i}=\prod_{i=0}^{3^{s-1}-1}\tep_i^{\,\,\ta_{i-j}}.
\]

Recall that our aim is to prove that the $\tdelta_j$'s fulfil no homogeneous non-trivial relation; so given any relation $\prod_j \tdelta_j^{e_j}=1$, with $\sum_j e_j=0$ we have to show that $e_j=0$ for any $0\leq j<3^{s-1}$. For this we consider the free $\Z$--module $M$ of rank $3^{s-1}$ with basis $f_0,f_1,\ldots,f_{3^{s-1}-1}$, the free $\Z$--module $N$ of rank $3^{s-1}$ with basis $d_0,d_1,\ldots,d_{3^{s-1}-1}$ and the homomorphisms
\[
\xymatrix{
N\ar[r]^\varphi & M\ar[r]^\pi & \calO^*\ar[r] &1  
}
\]
defined by $\pi(f_j)=\tep_j$, $\varphi(d_j)=\sum_{i=0}^{3^{s-1}-1}\ta_{i-j}f_i$ for any $j=0,1,\ldots,3^{s-1}-1$; note that $\pi\varphi(d_j)=\tdelta_j$ for any $j$. Let also $R$ be the submodule of $N$ of all elements $e=(e_j)_j=\sum_j e_jd_j$ such that $\varphi(e)\in\ker\pi$ and $\sum_j e_j=0$; our aim is equivalent to show that $R=0$.

Now the matrix $A$ given the map $\varphi$ in the basis $d_j$'s and $f_i$'s is $A=(\ta_{i-j})_{i,j}$, hence it is a circulant matrix of order $3^{s-1}$ and its content is $\sum_i\ta_i=2$ by the above formula expressing the $\delta_r$ in terms of the $\ep_d$. In particular $A$ has maximal rank by the previous Lemma, thus $\varphi$ is injective. Also, although $\tep_0=1$, the cyclotomic units $\tep_1,\tep_2,\ldots,\tep_{3^{s-1}-1}$ are multiplicative independent as proved by Kummer (see \cite{Ram} or \cite{Was}); so $\ker\pi=\Z f_0$.

Being $A$ circulant with content $2$, if $e\in R$ then $\sum_j(\varphi e)_j=2\sum_je_j=0$, hence $\varphi(e)=0$ using $\varphi(e)\in\ker\pi=\Z f_0$. But $\varphi$ is injective, so $e=0$ and this completes our proof.
\end{proof}

\begin{cor} Two tetrahedral space forms $X_\alpha$, $\alpha=\alpha_{\ell_0}\oplus\alpha_{\ell_1}\oplus\cdots\oplus\alpha_{\ell_{n-1}}$, and $X_\beta$, $\beta=\alpha_{\ell'_0}\oplus\alpha_{\ell'_1}\oplus\cdots\oplus\alpha_{\ell'_{m-1}}$, with $\ell_0\leq \ell_1\leq\cdots\leq \ell_{n-1}$ and $\ell'_0\leq \ell'_1\leq\cdots\leq \ell'_{m-1}$, have the same Reidemeister torsion if and only if $n=m$ and $\ell_j=\pm \ell'_j$ for all $j=0,1,\ldots,n-1$ (independent signs).
\end{cor}

\end{document}